\documentclass[11pt,a4paper]{amsart}
\usepackage{amsrefs}
\usepackage[T1]{fontenc}
\usepackage[latin1]{inputenc}
\usepackage[body={15cm, 24cm}]{geometry}
\usepackage{amsmath,amsfonts,amsthm}
\usepackage{graphicx,mathrsfs,amssymb,color}

\newtheorem{theorem}{\textbf{Theorem}}[section]
\newtheorem{proposition}[theorem]{\textbf{Proposition}}

\newtheorem{corollary}[theorem]{\textbf{Corollary}}

\theoremstyle{definition}
\newtheorem{definition}[theorem]{\textbf{Definition}}

\theoremstyle{remark}

\newtheorem{example}[theorem]{Example}
\newtheorem{remark}[theorem]{Remark}

\newcommand{\cP}{\mathcal{P}}

\newcommand{\R}{{\mathbb R}}
\newcommand{\N}{{\mathbb N}}
\newcommand{\supp}[1]{\operatorname{supp}#1}

\allowdisplaybreaks\let\inf\undefined
\DeclareMathOperator*{\inf}{inf\vphantom{p}}
\numberwithin{equation}{section}

\title[Superposition principle and schemes for MDEs]{Superposition principle and schemes for\\
Measure Differential Equations}

\author[F. Camilli]{Fabio Camilli}
\address{\hspace{-0.5em}\begin{tabular}{ll}Fabio Camilli:&``Sapienza'' Universit{\`a}  di Roma,\\&Dipartimento di Scienze di Base e Applicate per l'Ingegneria \\& 
Via Scarpa 16, I-00161 Rome, Italy.\end{tabular}}
\email{camilli@sbai.uniroma1.it}

\author[G. Cavagnari]{Giulia Cavagnari $^*$}
\address{\hspace{-0.5em}\begin{tabular}{ll}Giulia Cavagnari:& Politecnico di Milano,\\&Dipartimento di Matematica ``F. Brioschi''\\& 
Piazza Leonardo da Vinci, 32, I-20133 Milano, Italy.\end{tabular}}
\email{giulia.cavagnari@polimi.it}

\author[R. De Maio]{Raul De Maio}
\address{\hspace{-0.5em}\begin{tabular}{ll}Raul De Maio:&``Sapienza'' Universit{\`a}  di Roma,\\&Dipartimento di Scienze di Base e Applicate per l'Ingegneria \\& 
Via Scarpa 16, I-00161 Rome, Italy.\end{tabular}}
\email{raul.demaio@uniroma1.it}

\author[B. Piccoli]{Benedetto Piccoli}
\address{\hspace{-0.5em}\begin{tabular}{ll}Benedetto Piccoli:& Rutgers University - Camden,\\&Department of Mathematical Sciences\\ &311 N. 5th Street
Camden, NJ 08102, USA.\end{tabular}}
\email{piccoli@camden.rutgers.edu}

\keywords{Measure differential equations, superposition principle, measure-valued solution, probability vector fields}

\subjclass[2010]{35S99, 35F20, 35F25, 28A50}

\thanks{$^*$ Corresponding author: Giulia Cavagnari, \texttt{giulia.cavagnari@polimi.it}}

\begin{document}

\begin{abstract}
Measure Differential Equations (MDE) describe	the evolution of probability measures driven by probability  velocity fields, i.e. probability measures on the tangent bundle.  They are, on one side, a measure-theoretic generalization of ordinary differential equations; on the other side, they allow to describe  concentration and diffusion phenomena typical of kinetic equations. In this paper, we analyze some  properties of this  class of differential equations, especially highlighting their link with nonlocal continuity equations. We prove a representation result in the spirit of the Superposition Principle by Ambrosio-Gigli-Savar\'e, and we provide alternative schemes converging to a solution of the MDE, with a particular view to uniqueness/non-uniqueness phenomena.
\end{abstract}

\maketitle

\section{Introduction}
The theory of Measure Differential Equations (MDE in brief) has been recently introduced in  \cite{Piccoli}. A Cauchy problem for a  MDE is given by  
\begin{equation}\label{MDE_intro}
\left\{
\begin{array}{ll}
\dot \mu_t = V[\mu_t],\\
\mu_{t=0} = \mu_0,
\end{array}
\right.
\end{equation}
where   $\mu_0\in\cP(\R^d)$, the space of probability measures on $\R^d$,
and $V$ is a  probability vector field  (PVF in brief), i.e. a map  assigning to a probability measure $\mu\in\cP(\R^d)$  a (Young) probability measure $V[\mu]$ on the tangent bundle $T\R^d=\R^d\times\R^d$ such that the first projection of $V[\mu]$ is the base measure $\mu$ itself.

MDEs can be seen as a generalization of Ordinary Differential Equations (ODEs) to the space of measures compatible with the natural map associating to every point of $\mathbb R^d$ (more generally, a manifold) the Dirac measure centered at the same point. 
More precisely, in the particular case when $V[\mu]$ is concentrated on a graph, i.e. 
\begin{equation}\label{Vv}
V[\mu]=\mu \otimes \delta_{v(x)}
\end{equation}
for a  given Borel measurable vector field $v:\mathbb R^d\to\mathbb R^d$, then \eqref{MDE_intro} is equivalent to the continuity equation driven by $v$,
\begin{equation*}
\left\{
\begin{array}{ll}
\partial_t\mu_t+\mathrm{div}(v\mu_t)=0,\\
\mu_{t=0} = \mu_0,
\end{array}
\right.
\end{equation*}
showing also a first immediate link between MDEs and the continuity equation framework.
As proved in \cite{Piccoli}, Lipschitz continuity of $v$ guarantees existence and uniqueness of a Lipschitz semigroup of solutions to the corresponding MDE, which is linked to existence and uniqueness of (distributional) solutions of the corresponding transport/continuity equation.

The study of  linear and nonlinear transport equations, in the framework of weak measure solutions, has received a lot of attention in the recent time     (see \cite{AGS,CDMT1,CDMT2,ccr,rainbow,CMPbulgaro,PiccoliBOOK}). This theory is indeed relatively flexible to describe a large variety of phenomena,  as a continuum model for interacting particle systems (see e.g. \cites{G,J} for the derivation of mean-field equations as limit of interacting $N$-particles dynamics). The MDE approach can be seen as a further generalization of this technique, when the uncertainty affects not only the position of the particles, but also the law governing their evolution. As will be proved in Section \ref{sec:SP}, MDEs are in fact an alternative language to describe phenomena that fits into the study of nonlocal continuity equations, where the macroscopic evolution of the state of the system is ruled by a vector field that depends on the evolving state itself. In fact, in \cite{Piccoli} it is shown how MDEs can provide a unified model to study phenomena such as finite-speed diffusion and concentration (the latter using one-sided Lipschitz conditions), including conservation laws with discontinuous fluxes which may generate delta waves \cite{poupaud-rascle}. Finally MDE theory allows to prove mean-field type results for multi-particle systems.
For all these reasons, models aimed at describing collective motion \cites{ccr,VZ,BB}, such as pedestrian traffic \cite{PiccoliBOOK}, vehicular traffic \cite{CDMT1} or general multi-particle systems are natural applications of this study.
Moreover, a further development of the theory on MDEs in presence of a source term has been recently provided in \cite{MDE_PR}.

We recall that existence of weak measure solutions  to \eqref{MDE_intro} has been proved in  \cite{Piccoli} by means of an approximation scheme, called Lattice Approximate Solutions (LAS in brief). The scheme is obtained by discretizing the equation  in space, time and velocity  and moving convex combinations of Dirac masses  through the resulting discrete dynamical system. 

Aim of this paper is to provide a further analysis of \eqref{MDE_intro} to better understand certain properties regarding the solutions of the problem.  The first result is an equivalence relation with a nonlocal continuity equation dynamics, stated in Proposition \ref{prop:equivMDE-CE}, then we give an extension  of the Superposition Principle  by Ambrosio-Gigli-Savar\'e in the context of MDEs. We will provide a representation result for a solution of a MDE, similarly to what occurs for continuity equations with a local vector field (see \cite[Theorem 8.2.1]{AGS}), characterizing a (possibly not unique) solution of \eqref{MDE_intro} with a superposition of integral curves coming from a suitable underlying particle system.  In the same spirit, we  also provide a consistent probabilistic representation for the LAS scheme in \cite{Piccoli}.

In the second part of the paper, we consider alternative schemes converging to a weak solution of the MDE.
We first define a semi-discrete in time Lagrangian scheme for \eqref{MDE_intro}
and we prove that, up to subsequences, it converges to the same limit of the LAS scheme.
Moreover, we introduce another semi-discrete in time scheme obtained by taking the barycenter of the PVF at each time step, before moving the mass.
We show with an example that  this mean velocity scheme may converge to a different weak solution of \eqref{MDE_intro} with respect to the  LAS/Lagrangian schemes. This fact highlights the weak framework of the MDE theory, in what concerns uniqueness of solutions 
which, in general, is not expected (see Remark \ref{lack_uni} and the examples provided in Section \ref{sec:ex}). However, up to restrict the study to the class of solutions that can be obtained as limits of LASs (or with the semi-discrte in time Lagrangian scheme, see Corollary \ref{cor:uniqSemiLagr}), in \cite[Section 5]{Piccoli} the author prove the uniqueness of a Lipschitz semigroup associated to \eqref{MDE_intro} by prescribing the evolution of convex combinations of  Dirac measures for a small initial time.

The paper is organized as follows: in Section \ref{sec:prelim} we give some preliminaries on optimal transport and measure theory, recalling the MDE setting and the definition of the LAS scheme introduced in \cite{Piccoli}; in Section \ref{sec:SP} we highlight the correspondences between weak solutions of MDE and (distributional) solutions of a suitable continuity equation, we thus exploit a Superposition Principle for MDEs and a probabilistic representation construction for the LASs; in Section \ref{sec:SDS} we provide a Lagrangian approximation scheme; in Section \ref{sec:MVP}, we present another approximating scheme converging to a different solution of \eqref{MDE_intro} and finally, in Section \ref{sec:ex}, we discuss some clarifying examples, where we make use of the previously defined schemes and of what analyzed in Section \ref{sec:SP}.

\section{Preliminaries and first results}\label{sec:prelim}
We recall some preliminary definitions and results
(we  address the reader to \cites{AGS,Santa,Villani} as relevant resources regarding optimal transport and measure theory).
Given a complete separable metric space $X$, we denote by $\cP(X)$  the set of Borel probability measures on $X$, by  $\cP_p(X)$  the subset of $\cP(X)$ whose elements have finite $p$-moment and by $\cP_c(X)$  the subset of $\cP_p(X)$ whose elements have compact support.  We endow the set $\cP_p(X)$ with  the $p$-Wasserstein distance $W^X_p$ which makes $\cP_p(X)$ a complete separable metric space. On $\cP_c(X)$, we consider the metric $W_1^X$. In the case $p=1$, we recall a special duality formula, called \emph{the Kantorovich-Rubinstein duality}
\begin{equation*}
W^X(\mu,\nu)=\sup\left\{\int_X f\,d(\mu-\nu)\,:\,f:X\to\R,\,\mathrm{Lip}(f)\le 1\right\}.
\end{equation*}

Referring e.g. to \cite[Section 5.2]{AGS} for the definition of push-forward, $r\sharp\mu$, of a probability measure $\mu\in\cP(X)$ through a Borel map $r:X\to Y$, in the following we recall an important disintegration result (see \cite[Theorem 5.3.1]{AGS}).

\begin{theorem}[Disintegration]\label{thm:disintegr}
Let $\boldsymbol X,X$ be complete separable metric spaces, $\boldsymbol\mu\in\cP(\boldsymbol X)$ and $r:\boldsymbol X\to X$ be a Borel map. Then there exists a $r\sharp\boldsymbol\mu$-a.e. uniquely determined Borel family of probability measures $\{\mu_x\}_{x\in X}\subset\cP(\boldsymbol X)$ such that $\mu_x(\boldsymbol X\setminus r^{-1}(x))=0$ for $r\sharp\boldsymbol\mu$-a.e. $x\in X$. Furthermore
\[\int_{\boldsymbol X}f(z)\,d\boldsymbol\mu(z)=\int_X\int_{r^{-1}(x)}f(z)\,d\mu_x(z)\,d(r\sharp\boldsymbol\mu)(x),\]
for any bounded Borel map $f:\boldsymbol X\to\mathbb{R}$. We will write $\boldsymbol\mu=(r\sharp\boldsymbol\mu)\otimes\mu_x$.
\end{theorem}
\begin{remark}\label{rmk:disintegr}
As pointed out in \cite[Section 5.3]{AGS}, if $\boldsymbol X=X\times Y$ and $r^{-1}(x)\subseteq \{x\}\times Y$ for all $x\in X$, then we can identify each measure $\mu_x\in\cP(X\times Y)$ with a measure defined on $Y$. We will make a strong use of this result throughout the paper.
\end{remark}

We recall now the definition of convolution between measures and product with a coefficient $a\in\R$. We denote with $\chi_A$ the characteristic function of $A\subseteq\R^d$.

\begin{definition}[Convolution]\label{def:conv}
We define \emph{the convolution operator} $\oplus:\cP(\mathbb R^d)\times\cP(\mathbb R^d)\to\cP(\mathbb R^d)$ by $(\mu\oplus\nu)(B):=\int_{\R^d\times\R^d}\chi_B(x+y)\,d\nu(y)\,d\mu(x)$, for any Borel set $B\subseteq\R^d$. Equivalently we may define
\[\langle\mu\oplus\nu,f\rangle:=\int_{\R^d\times\R^d}f(x+y)\,d\nu(y)\,d\mu(x),\]
for any $\mu\oplus\nu$-integrable Borel function $f:\R^d\to\R$.\\
We define \emph{the product operator} $\cdot:\R\times\cP(\R^d)\to\cP(\R^d)$, $(a,\mu)\mapsto a\cdot\mu$, by 
\begin{equation}\label{eq:prodDef}
(a\cdot\mu)(B):=\int_{\R^d}\chi_B(ax)\,d\mu(x),
\end{equation}
for any Borel set $B\subseteq\R^d$.
\end{definition}
\begin{remark}
Observe that $\cP_c(\R^d)$ is closed w.r.t. convolution and product operators. In particular, as pointed out in \cite[Section 6.1]{Piccoli} we have that the operation $\oplus$ defines a monoid structure over $\cP_c(\R^d)$.
\end{remark}

Before setting the theory to study Measure Differential Equations, let us recall briefly the definition of solution for a continuity equation (see e.g. \cite[Section 8.1]{AGS}). These two notions will be compared in Section \ref{sec:SP}.
\begin{definition}[Continuity equation]
Given $T>0$, a Borel family of probability measures $\boldsymbol{\mu}=\{\mu_{t}\}_{t\in [0,T]}\subseteq\mathcal P(\mathbb R^d)$ and
a Borel map $w:[0,T]\times\mathbb R^d\to\mathbb R^d$, $w_t(x)=w(t,x)$, such that
$\displaystyle\int_{\mathbb R^d}|w_t(x)|\,d\mu_t(x)$ belongs to $L^1([0,T])$,
we say that $\boldsymbol{\mu}$ solves the \emph{continuity equation}
\begin{equation}\label{eq:cont}
\partial_t\mu_t+\mathrm{div}(w_t\mu_t)=0,
\end{equation}
if for every $f\in C^{\infty}_c(\mathbb R^d)$ there holds
\begin{equation}\label{eq:contdistr}
\dfrac{d}{dt}\langle \mu_t,f\rangle=\int_{\mathbb R^d}\left(\nabla f(x)\cdot w_t(x)\right)\,d\mu_t(x),
\end{equation}
in the sense of distributions on $(0,T)$.
\end{definition}

\begin{remark}\label{rmk:uniqNLCE}
We can give a meaning to a solution of \eqref{eq:cont} also when the vector field $w$ is nonlocal depending on the solution itself, i.e. $w:[0,T]\times\mathbb R^d\times\mathcal P(\mathbb R^d)\to\mathbb R^d$, $w(t,x,\mu)=w_t[\mu](x)$. In this case, \eqref{eq:cont} reads
\begin{equation}\label{eq:contNL}
\partial_t\mu_t+\mathrm{div}(w_t[\mu_t]\mu_t)=0,
\end{equation}
to be understood in the distributional sense. Given $\mu_0\in\cP_p(\R^d)$, $p>1$, we recall that if the vector field $w$ is Lipschitz continuous in $(x,\mu)$ uniformly w.r.t. $t$, i.e. there exists $L>0$ such that
\begin{equation}\label{eq:lipwbar}
|w(t,x,\mu^1)-w(t,y,\mu^2)|\le L\left(|x-y|+W^{\R^d}(\mu^1,\mu^2)\right) \quad\forall x,y\in\R^d,\,\forall \mu^1,\mu^2\in \cP_1(\R^d),
\end{equation}
then there exists a unique solution $\boldsymbol\mu\in C([0,T];\cP_1(\R^d))$ to \eqref{eq:contNL} with initial datum $\mu_{t=0}=\mu_0$.
This result is proved and detailed for example in \cite[Theorem 6.1]{Orrieri}. Similar conditions are used in \cite{PR} to prove existence and uniqueness of solutions to \eqref{eq:contNL} with initial datum $\mu_0$, in the class of measures that are absolutely continuous w.r.t. the Lebesgue measure. Under assumptions granting uniqueness, in \cite{PR} the authors construct the solution through converging numerical schemes.
\end{remark}

\subsection{Recalls on Measure Differential Equations}
In this section we recall some basic definitions introduced in \cite{Piccoli} that are at the base of the investigations proposed in this paper. Throughout the work, we denote with $T\R^d=\R^d\times\R^d$ the tanglent bundle of $\R^d$ keeping in mind that when $(x,v)\in T\R^d$, then $x$ should be interpreted as a position while $v$ as a velocity. We denote with $\pi^1:T\R^d\to\R^d$ the projection to the first component, i.e. $\pi^1(x,v)=x$.

\begin{definition} 
A \emph{probability vector field} (PVF) is a map $V:\cP(\R^d)\to\cP(T\R^d)$ consistent with the projection on the first component, i.e. $\pi_1\sharp V[\mu]=\mu$.
\end{definition}
By Theorem \ref{thm:disintegr} and Remark \ref{rmk:disintegr}, we can write $V[\mu]=\mu\otimes\nu_x[\mu]$ for a $\mu$-a.e. uniquely determined Borel family of probability measures $\{\nu_x[\mu]\}_{x\in\R^d}\subseteq\cP(\R^d)$ defined on the fibers $T_x\R^d=\R^d$.\\
Let $T>0$. Given $\mu_0\in\cP(\R^d)$ and a PVF $V$, we consider the following Cauchy problem
\begin{equation}\label{MDE}
\left\{
\begin{array}{ll}
\dot \mu_t = V[\mu_t],\quad t\in(0,T)\\
\mu_{t=0} = \mu_0,
\end{array}
\right.
\end{equation}
where the nonlocal dynamics is called \emph{Measure Differential Equation} (MDE). A solution to this problem has to be interpreted as follows.

\begin{definition}\label{def:sol}
A solution of \eqref{MDE} is a map $\boldsymbol\mu: [0,T] \to \cP(\R^d), t\mapsto\mu_t$ such that $\mu_{t=0} = \mu_0$ and that satisfies for a.e. $t \in [0,T]$
\begin{equation}\label{eq:MDEdistr}
\frac{d}{dt}\langle \mu_t, f\rangle = \int_{T\R^d}(\nabla f(x)\cdot v)\,dV[\mu_t](x,v),
\end{equation}
for any $f\in C^\infty_c(\R^d)$ such that the right-hand side is defined for a.e. $t$, the map $t\mapsto\int_{T\R^d}(\nabla f(x)\cdot v)\,dV[\mu_t](x,v)$ belongs to $L^1([0,T])$, and the map $t\mapsto\int_{\R^d}f\,d\mu_t$ is absolutely continuous. 
Equivalently,
\begin{equation*}
\langle \mu_t - \mu_0, f\rangle = \int_0^t \int_{T\R^d}(\nabla f(x)\cdot v)\,dV[\mu_s](x,v)\,ds,\quad \forall f \in C^\infty_c(\R^d).
\end{equation*}
\end{definition}
\begin{remark}
Notice that, in the trivial case when $V[\mu]=\mu\otimes\delta_{v(x)}$ for some Borel vector field $v:\R^d\to\R^d$, the MDE in \eqref{MDE} reduces to the continuity equation $\partial_t\mu_t+\mathrm{div}\,(v\mu_t)=0$ (see \cite[Section 6]{Piccoli}).
\end{remark}
We stress that $V[\mu]$ is a probability measure on $T\R^d$ where the components of its elements $(x,v)$ represent, respectively, the position and the infinitesimal displacement. We  recall another notion to measure distances between PVFs introduced in \cite{Piccoli}. 
\begin{definition}
Given  $V_i\in\cP_c(T\R^d)$, $i=1,2,$ and denoted by $\mu_i := \pi_1\sharp V_i$ the marginal of $V_i$, we define
\begin{equation*}
\mathcal{W}(V_1, V_2) = \inf\left\{ \int_{T\R^d \times T\R^d} |v - w|dT(x,v,y,w) : T \in \Pi(V_1, V_2), \pi_{13}\sharp T \in \Pi_{opt}(\mu_1, \mu_2) \right\},
\end{equation*}
where $ \Pi(V_1, V_2)$ is the set of all the transference plans from $V_1$ to $V_2$ and $\Pi_{opt}(\mu_1, \mu_2) $ is the set
of the optimal transference plans from $\mu_1$ to $\mu_2$, and $\pi_{13}:(T\R^d)^2\to (\R^d)^2$, $(x,v,y,w)\mapsto (x,y)$.
\end{definition}
The object $\mathcal{W} $ computes  the minimal displacements of the fiber components   assuming that marginals
$\mu_i$ are transported in an optimal way. It is important to notice that $\mathcal{W}$ is not a metric since it can vanish for distinct elements in $\cP_c(T\R^d)$. Moreover, it is easy to verify that 
\begin{equation*}
W^{T\R^d}(V_1, V_2) \leq \mathcal{W}(V_1, V_2) + W^{\R^d}(\mu_1, \mu_2).
\end{equation*}

\medskip

Considering the problem set in $\cP_c(\R^d)$, we recall here the main assumptions required to have existence and convergence of approximation schemes for solutions of an MDE (see \cite{Piccoli}).
\begin{enumerate}
\item[$(H1)$] \textbf{Sublinearity:} there exists a constant $C>0$ such that for all $\mu\in\cP_c(\R^d)$,
\[\sup_{(x,v) \in \text{supp}V[\mu]}|v| \leq C(1 + \sup_{x \in \text{supp}\mu}|x|);\]
\item[$(H2)$] \textbf{Continuity of PVF}: the map $V:\cP_c(\R^d) \to \cP_c(T\R^d) $ is continuous.
\end{enumerate}
As shown in Theorem \ref{thm:convMV} (and in \cite[Theorem 3.1]{Piccoli}) these hypothesis are sufficient to conclude existence of solutions to \eqref{MDE}. When specified, we will require also the following Lipschitz regularity assumption
\begin{enumerate}
\item[$(H3)$] \textbf{Local lipschitzianity in $\mu$-variable}:  $V$ is locally Lipschitz, in particular for every $R>0$ there exists a constant $L=L(R)>0$ such that
\[\mathcal{W}(V[\mu], V[\nu]) \leq L\cdot W^{\R^d}(\mu, \nu),\]
for every $\mu, \nu \in \cP_c(\R^d)$ such that $\supp(\mu),\supp(\nu) \subset B(0,R)$, the open ball of radius $R$ centered in $0\in\R^d$.
\end{enumerate}

\medskip

In the following, we recall the scheme provided in \cite{Piccoli} that has been used in order to prove existence of solutions to \eqref{MDE}. Let us start introducing some notation. 
Let $T>0$. For $N\in\N\setminus\{0\}$, set
\[\Delta_N=\frac{T}{N},\quad \Delta_N^v=\frac{1}{N},\quad \Delta_N^x=\Delta_N^v\,\Delta_N=\frac{T}{N^2}\]
be respectively the time, the velocity and the space-step sizes, noticing that, differently from \cite{Piccoli}, we set the time step size to $T/N$ in place of $1/N$, for our convenience. Considering the corresponding grid in $[0,T]\times[-N,N]^d\times[-TN,TN]^d$, we denote by $x_i$ the discretization points in space, and by $v_i$ the discretization points for the space of velocities. We now build some objects aiming at providing a discrete approximation for $\mu\in\cP_c(\R^d)$ and $V[\mu]\in\cP_c(T\R^d)$ by concentrating the mass on the points of the grid.
Denoting with $Q=([0,\frac{T}{N^2}[)^d$ and $Q'=([0,\frac{1}{N}[)^d$, we define
\[\mathcal A_N^x(\mu):=\sum_i m_i^x(\mu)\,\delta_{x_i},\quad \mathcal A_N^v(V[\mu]):=\sum_i\sum_j m_{ij}^v(V[\mu])\,\delta_{(x_i,v_j)},\]
where $m_i^x(\mu):=\mu(x_i+Q)$ and $m_{ij}^v(V[\mu]):=V[\mu](\{(x_i,v)\,:\,v\in v_j+Q'\})$.
Notice that given $\mu\in\cP_c(\R^d)$, for $N$ sufficiently large we have
\begin{equation}\label{eq:estimLAS}
W^{\R^d}(\mathcal A_N^x(\mu),\mu)\le\Delta_N^x,\quad W^{T\R^d}(\mathcal A_N^v(V[\mu]),V[\mu])\le\Delta_N^v.
\end{equation}

\begin{definition}[Lattice Approximate Solution (LAS)]\label{def:LAS}
Let $V$ be a PVF satisfying $(H1)$. Given $\mu_0\in\cP_c(\R^d)$, $T>0$ and $N\in\N$, the \emph{Lattice Approximate Solution (LAS)} $\boldsymbol\mu^N:[0,T]\to\cP_c(\R^d)$, $t\mapsto\mu_t^N$, is defined, by recursion, as follows
\begin{align*}
\mu_0^N&=\mathcal A_N^x(\mu_0)\\
\mu^N_{k+1}&=\mu^N((k+1)\Delta_N)=\sum_{ij}m_{ij}^v(V[\mu^N(k\Delta_N)])\,\delta_{x_i+\Delta_N\,v_j},
\end{align*}
notice that $\supp(\mu_k^N)$ is contained on the space grid. By time-interpolation we can define $\boldsymbol\mu^N$ for all times as
\begin{equation*}
\boldsymbol\mu^N(k\Delta_N+t)=\sum_{ij}m_{ij}^v(V[\mu^N(k\Delta_N)])\,\delta_{x_i+t\,v_j}.
\end{equation*}
\end{definition}

\medskip

Notice that, thanks to the growth assumption $(H1)$ on $V$ and since $\mu_0\in\cP_c(\R^d)$, then the support of $\mu^N_t$ is uniformly bounded (see \cite[Lemma 3.3]{Piccoli}), i.e. there exists $R=R(T,C,\mu_0)>0$ such that
\begin{equation}\label{eq:growthLAS}
\text{supp}\,\mu^N_t\subset B(0,R), \quad\forall t\in[0,T],	\,\forall N\in\N\setminus\{0\}.
\end{equation}

We address the reader to \cite{Piccoli} for results granting the convergence of the LAS scheme to a solution of \eqref{MDE}.

\medskip

We recall here a definition, used in \cite{Piccoli} to derive the uniqueness of solutions to \eqref{MDE} at the level of the semigroup, when restricting the analysis to a certain class of trajectories. This will be resumed later on in Section \ref{sec:SDS}.

\begin{definition}\label{def:semigroup}
 A Lipschitz semigroup for \eqref{MDE} is a map $S: [0,T] \times \cP_c(\R^d) \to \cP_c(\R^d)$ such that for every $\mu, \eta \in \cP_c(\R^d)$ and $t,s \in [0,T]$ the following holds:
\begin{itemize}
\item[(i)] $S_0\mu =\mu$ and $S_t S_s \mu = S_{t+s}\mu$;
\item[(ii)] the map $t \to S_t \mu$ is a solution of \eqref{MDE};
\item[(iii)] for every $R>0$ there exists $C=C(R)>0$ such that if $\supp\,(\mu), \supp\,(\eta) \subset B(0,R)$ then:
$$
\left\lbrace 
\begin{array}{l}
 \supp\, (S_t \mu) \subset B(0, e^{Ct}(R + 1)),\\
 W^{\R^d}(S_t \mu, S_t\eta) \leq e^{Ct}W^{\R^d}(\mu, \eta),\\
 W^{\R^d}(S_t\mu, S_s\mu) \leq C| t - s |.
\end{array}
\right.$$
\end{itemize}
\end{definition}


\section{A Superposition Principle for MDEs}\label{sec:SP}

We start this section by exploiting the definition of solutions for an MDE comparing it with solutions of a continuity equation. We see that, by definition, in order for $\boldsymbol\mu$ to be a solution of an MDE, it is sufficient to solve a nonlocal continuity equation driven by the barycenter of what $V[\mu]$ prescribes on the fibers $T_x\R^d$ along the solution.
In this sense, MDEs can be viewed as an alternative language to nonlocal continuity equations.
\begin{proposition}\label{prop:equivMDE-CE}
Let $T>0$, $V:\cP(\R^d)\to\cP(T\R^d)$ be a PVF satisfying
\begin{equation}\label{eq:hypintV}
\int_0^T\int_{T\R^d}|v|\,d V[\mu_t](x,v)\,dt<+\infty.
\end{equation}
Then $\boldsymbol\mu: [0,T] \to \cP(\R^d), t\mapsto\mu_t$ is a solution of the MDE $\dot\mu_t=V[\mu_t]$ if and only if $\boldsymbol\mu$ is a solution of the nonlocal continuity equation \eqref{eq:contNL} for a Borel vector field $w:\mathbb R^d\times\mathcal P(\mathbb R^d)\to\mathbb R^d$ defined by
\begin{equation}
w[\mu](x):=\int_{\pi_1^{-1}(x)}v\,d\nu_x[\mu](v),\quad \mu_t\textrm{-a.e. }x\in\mathbb R^d,
\end{equation}
for $\mu$-a.e. $x\in\mathbb R^d$, 
where we denoted with $\nu_x[\mu]$ the disintegration of $V[\mu]$ w.r.t. $\pi_1$.
\end{proposition}

\begin{proof}
The proof follows by disintegration technique. Let us first notice that by \eqref{eq:hypintV}
\begin{align*}
\int_0^T\int_{\mathbb R^d}|w[\mu_t](x)|\,d\mu_t(x)\,dt\le\int_0^T\int_{\mathbb R^d}\int_{\mathbb R^d}|v|\,d\nu_x[\mu_t](v)\,d\mu_t(x)\,dt<+\infty.
\end{align*}
The conclusion comes straightforwardly noting that \eqref{eq:MDEdistr} can be written in terms of \eqref{eq:contdistr} as follows
\begin{align*}
\dfrac{d}{dt}\langle \mu_t,f\rangle&= \int_{T\R^d}(\nabla f(x)\cdot v)\,dV[\mu_t](x,v)\\
&=\int_{\mathbb R^d}\int_{\pi_1^{-1}(x)}(\nabla f(x)\cdot v)\,d\nu_x[\mu_t](v)\,d\mu_t(x)\\
&=\int_{\mathbb R^d}\left(\nabla f(x)\cdot w[\mu_t](x)\right)\,d\mu_t(x).
\end{align*}
\end{proof}

\begin{remark}\label{lack_uni}
As briefly discussed in Remark \ref{rmk:uniqNLCE}, we know that we can guarantee  uniqueness of solutions to the nonlocal continuity equation \eqref{eq:contNL} (with fixed initial datum) under the Lipschitz assumption \eqref{eq:lipwbar} on the driving velocity field $w$.
However, 
when we set the problem in the framework of MDEs, provide such condition is not an easy task, even having the equivalence just proved in Proposition \ref{prop:equivMDE-CE}. Indeed by Proposition \ref{prop:equivMDE-CE}, if we want to have (Lipschitz) continuity of $w$ w.r.t. $x\in\mathbb R^d$, we should assume the disintegration $x\mapsto\nu_x[\mu]$ of the PVF $V[\mu]$ w.r.t. the projection to the base $\pi_1$ to be (Lipschitz) continuous, while in general it is just Borel measurable (see Theorem \ref{thm:disintegr}).

Thus, in general, uniqueness of a weak solution for \eqref{MDE} is not expected (see   \cite[Example 3]{Piccoli}) even under assumptions $(H1)-(H3)$. We will discuss further examples in Section \ref{sec:ex}.
\end{remark}

\medskip

Inspired by the previous result,
we show how to construct a Superposition Principle (see \cite[Theorem 8.2.1]{AGS} for the local continuity equation dynamics) adapted to the language of MDEs. The procedure is similar to the one used in \cite{CMPbulgaro}, where the authors provide a representation result for solutions of a continuity equation associated with Carath\'eodory solutions of a differential inclusion. This result, proved in \cite{CMPbulgaro}, is exploited in   \cite{rainbow}, where the authors study optimal control problems in the space of probability measures with microscopic dynamics ruled, precisely, by a differential inclusion. \\We split the statement into two parts. In the first part, we see that any measure $\boldsymbol\eta\in\cP(\Gamma_{[0,T]})$, concentrated on curves that follow a given PVF $V$ in integral average, generates a solution of the MDE.\\
For $I\subseteq\R$ interval, we denote by $\Gamma_I$  the set of continuous curves from $I$ to $\R^d$ and  by 
$e_t$  the evaluation operator $e_t:\Gamma_I\to\R^d$, $\gamma\mapsto\gamma(t)$, for $t\in I$, while $AC(I;\R^d)$ is the space of absolutely continuous curves from $I$ to $\R^d$. 

\begin{theorem}[Superposition Principle for MDEs - Part I]\label{thm:SP1}
Let $T>0$, $V:\cP(\R^d)\to\cP(T\R^d)$ be a PVF, $\mu_0\in\cP(\R^d)$. Let $\boldsymbol\eta\in\cP(\Gamma_{[0,T]})$ be concentrated on the set of curves $\gamma\in\Gamma_{[0,T]}$ such that the following conditions hold
\begin{itemize}
\item[$(i)$] $\gamma\in AC([0,T];\R^d)$;
\item[$(ii)$] for a.e. $t\in[0,T]$ and for $e_t\sharp\boldsymbol\eta$-a.e. $y\in\R^d$ we have
\[\int_{e^{-1}_t(y)} \dot\gamma(t)\,d\eta_{t,y}(\gamma)=\int_{\pi^{-1}_1(y)} v\,d \nu_y[e_t\sharp\boldsymbol\eta](v),\]
where $\eta_{t,y}$ is the disintegration of $\boldsymbol\eta$ w.r.t. $e_t$, and $\nu_y[e_t\sharp\boldsymbol\eta]$ is the disintegration of $V[e_t\sharp\boldsymbol\eta]$ w.r.t. the projection to the base $\pi_1$;
\item[$(iii)$] $\displaystyle\int_0^T\int_{\Gamma_{[0,T]}}|\dot\gamma(t)|\,d\boldsymbol\eta(\gamma)\,dt<+\infty$ and $\displaystyle\int_0^T\int_{T\mathbb R^d}|v|\,dV[e_t\sharp\boldsymbol\eta](x,v)\,dt<+\infty$.
\end{itemize}
Then, denoted with $\mu_t:=e_t\sharp\boldsymbol\eta$, we have that $\boldsymbol\mu=\{\mu_t\}_{t\in[0,T]}\subseteq\cP(\R^d)$ is a solution of the MDE system \eqref{MDE}.
\end{theorem}
\begin{proof}
Let us consider any $f\in C^\infty_c(\R^d)$.
First, we check that $\int_{T\R^d}(\nabla f(x)\cdot v)\,d V[e_s\sharp\boldsymbol\eta](x,v)$ is defined for almost every $s\in[0,T]$. Indeed, immediately by hypothesis $(iii)$,
\begin{align*}
\int_{T\R^d}(\nabla f(x)\cdot v)\,d V[e_s\sharp\boldsymbol\eta](x,v)\le\|\nabla f\|_\infty\int_{T\R^d}|v|\,d V[e_s\sharp\boldsymbol\eta](x,v)<+\infty
\end{align*}
for a.e. $s\in[0,T]$.
Thus, we also have that $s\mapsto\int_{T\R^d}(\nabla f(x)\cdot v)\,d V[e_s\sharp\boldsymbol\eta](x,v)$ belongs to $L^1([0,T])$.

Secondly, the map $t\mapsto\int_{\R^d}f\,d\mu_t$ is absolutely continuous. Indeed, for $0\le s<t\le T$ we have
\begin{align*}
\left|\int_{\R^d}f\,d\mu_s-\int_{\R^d}f\,d\mu_t\right|&\le\int_s^t\int_{\Gamma_{[0,T]}}|\nabla f(\gamma(\tau))\cdot\dot\gamma(\tau)|\,d\boldsymbol\eta(\gamma)\,d\tau\\
&\le \int_s^t\int_{\Gamma_{[0,T]}}|\nabla f(\gamma(\tau))|\cdot|\dot\gamma(\tau)|\,d\boldsymbol\eta\,d\tau\\
&\le\|\nabla f\|_\infty\int_s^t\int_{\Gamma_{[0,T]}}|\dot\gamma(\tau)|\,d\boldsymbol\eta\,d\tau<+\infty,
\end{align*}
thanks to hypothesis $(iii)$.

Lastly, for a.e. $t\in[0,T]$,
\begin{align*}
\frac{d}{dt}\int_{\R^d} f(x)\,d\mu_t(x)&=\int_{\Gamma_{[0,T]}}\nabla f(\gamma(t))\cdot\dot\gamma(t)\,d\boldsymbol\eta(\gamma)\\
&=\int_{\R^d}\nabla f(y)\int_{e_t^{-1}(y)}\dot\gamma(t)\,d\eta_{t,y}(\gamma)\,d\mu_t(y)\\
&=\int_{\R^d}\nabla f(y)\int_{\pi_1^{-1}(y)}v\,d \nu_y[\mu_t](v)\,d\mu_t(y)\\
&=\int_{T\R^d}\nabla f(x)\cdot v\,d V[\mu_t](x,v),
\end{align*}
where we used hypothesis $(ii)$.
\end{proof}

\begin{remark}
We observe that the second request in item $(iii)$ can be satisfied assuming hypothesis $(H1)$ for the PVF $V$ together with the hypothesis
\[\int_0^T\int_{\Gamma_{[0,T]}}\sup_{\gamma\in\mathrm{supp}\,\boldsymbol\eta}|\gamma(t)|\,d\boldsymbol\eta(\gamma)\,dt<+\infty.\]
\end{remark}
\medskip

Let us now pass to the other implication.
In the second part of the statement that we are going to see, we want to prove the existence of a probabilistic representation $\boldsymbol\eta$ starting from a solution $\boldsymbol\mu$ of the MDE system \eqref{MDE} with given PVF $V$. In this general framework, this can be easily provided by disintegration technique and thanks to 
the Superposition Principle in \cite[Theorem 8.2.1]{AGS}.

\begin{theorem}[Superposition Principle for MDEs - Part II]
Let $T>0$, $V:\cP(\R^d)\to\cP(T\R^d)$ be a PVF, $\mu_0\in\cP(\R^d)$. Let $\boldsymbol\mu=\{\mu_t\}_{t\in[0,T]}\subseteq\cP(\R^d)$ be an absolutely continuous solution of the MDE system \eqref{MDE} for a PVF $V:\cP(\R^d)\to\cP(T\R^d)$ satisfying
\begin{equation*}
\int_0^T\int_{T\R^d}|v|^p\,d V[\mu_t](x,v)\,dt<+\infty, \textrm{ for some }p>1.
\end{equation*}
Then there exists a probability measure $\boldsymbol\eta\in\cP(\Gamma_{[0,T]})$ such that 
\begin{itemize}
\item[$(i)$] $\boldsymbol\eta$ is concentrated on curves $\gamma$ such that $\gamma\in AC([0,T];\R^d)$ is a solution of the ODE $\dot\gamma(t)=w_t(\gamma(t))$ for a.e. $t\in(0,T)$, where $w:[0,T]\times\mathbb R^d\to\mathbb R^d$, $(t,y)\mapsto w_t(y)=\displaystyle\int_{\pi_1^{-1}(y)}v\,d\nu_y[\mu_t](v)$ for a.e. $t$ and $\mu_t$-a.e. $y$;
\item[$(ii)$] $\mu_t=e_t\sharp\boldsymbol\eta$ for any $t\in[0,T]$.
\end{itemize}
Straightforwardly, if $\boldsymbol{\hat\eta}\in\cP(\Gamma_{[0,T]})$ realizes $(i)$ and $(ii)$, then for a.e. $t\in[0,T]$ and for $\mu_t$-a.e. $y\in\R^d$ we have
\[\int_{e^{-1}_t(y)} \dot\gamma(t)\,d\hat\eta_{t,y}(\gamma)=\int_{\pi^{-1}_1(y)} v\,d \nu_y[\mu_t](v),\]
where $\hat\eta_{t,y}$ is the disintegration of $\boldsymbol{\hat\eta}$ w.r.t. $e_t$, and $\nu_y[\mu_t]$ is the disintegration of $V[\mu_t]$ w.r.t. the projection on the first component $\pi_1$.
\end{theorem}
\begin{proof}
Let us take any $f\in C^\infty_c(\R^d)$.
Let $\boldsymbol\mu=\{\mu_t\}_t$ be as in the statement. Then, by Definition \ref{def:sol} and by disintegrating $V$ w.r.t. the projection map $\pi_1$,
\begin{align*}
\frac{d}{dt}\int_{\R^d}f(y)\,d\mu_t(y)&=\int_{T\R^d}\nabla f(y)\cdot v\,d V[\mu_t](y,v)\\
&=\int_{\R^d}\nabla f(y)\cdot\int_{\pi_1^{-1}(y)}v\,d \nu_y[\mu_t](v)\,d\mu_t(y)
\end{align*}
for a.e. $t\in[0,T]$.

By denoting with $w:[0,T]\times\R^d\to\R^d$, $w_t(y):=\int_{\pi_1^{-1}(y)}v\,d \nu_y[\mu_t](v)$ for a.e. $t$ and $\mu_t$-a.e. $y$, we notice that $w$ is a Borel vector field by disintegration property (Theorem \ref{thm:disintegr}) and that $\boldsymbol\mu$ is a solution of the continuity equation $\partial_t\mu_t+\mathrm{div}(w_t\mu_t)=0$.

Moreover, by Jensen's inequality, we get
\begin{align*}
\int_0^T\int_{\R^d}|w_t(y)|^p\,d\mu_t(y)\,dt&=\int_0^T\int_{\R^d}\left|\int_{\pi_1^{-1}(y)}v\,d \nu_y[\mu_t](v)\right|^p d\mu_t(y)\,dt\\
&\le\int_0^T\int_{T\R^d}|v|^p\,d V[\mu_t](y,v)\,dt<+\infty,\\
\end{align*}
thanks to the hypothesis on $V$.
We can thus apply the classical Superposition Principle in \cite[Theorem 8.2.1]{AGS} to get $(i-ii)$.

Let now $\boldsymbol{\hat\eta}$ be as in the statement, hence $\dot\gamma(t)=w_t(\gamma(t))$ for a.e. $t\in(0,T)$ and for $\boldsymbol{\hat\eta}$-a.e. $\gamma$. Last property is strainghtforward by disintegration and definition of $w$, indeed for any test function $\varphi\in C^\infty_c([0,T]\times\mathbb R^d)$,
\begin{align*}
\int_0^T\int_{\Gamma_{[0,T]}}\varphi(t,y)\,\dot\gamma(t)\,d\boldsymbol{\hat\eta}(\gamma)\,dt&=\int_0^T\int_{\R^d}\varphi(t,y)\int_{e_t^{-1}(y)}\dot\gamma(t)\,d\hat\eta_{t,y}(\gamma)\,d\mu_t(y)\,dt\\
&=\int_0^T\int_{\R^d}\varphi(t,y)\int_{e_t^{-1}(y)}w_t(y)\,d\hat\eta_{t,y}(\gamma)\,d\mu_t(y)\,dt\\
&=\int_0^T\int_{\R^d}\varphi(t,y)\int_{\pi_1^{-1}(y)}v\,d \nu_y[\mu_t](v)\,d\mu_t(y)\,dt.
\end{align*}
\end{proof}

We now complete the analysis concerning the connection between Superposition Principle and   MDEs by giving an example of an explicit and consistent construction for a probabilistic representation of the LAS scheme.\\
Let $\boldsymbol\mu\subseteq\cP(\R^d)$ be a solution of the MDE system \eqref{MDE} obtained as uniform-in-time limit of LASs $\{\boldsymbol\mu^N\}_{N\in\N}$. We now construct a probabilistic representation for $\boldsymbol\mu$ that is concentrated on uniform limits of the trajectories $\gamma_{i,j}:[0,T]\to\mathbb R^d$, $\gamma_{i,j}(t)=x_i+t v_j$, where the LASs $\boldsymbol\mu^N$ are concentrated.

\medskip

Let us start by fixing some notation. Given $I_1, I_2\subset\R$ nonempty and compact intervals, with $\max I_1=\min I_2$, we define
\begin{enumerate}
	\item the \emph{set of compatible trajectories}
	\[\mathcal D_{I_1,I_2}:=\left\{(\gamma_1,\gamma_2)\in\Gamma_{I_1}\times\Gamma_{I_2}\,:\,\gamma_1(\max I_1)=\gamma_2(\min I_2)\right\};\]
	\item the \emph{concatenation} $\gamma_1\star\gamma_2\in\Gamma_{I_1\cup I_2}$ of curves $\gamma_1\in\Gamma_{I_1}$, $\gamma_2\in\Gamma_{I_2}$, with $\gamma_1(\max I_1)=\gamma_2(\min I_2)$, is a map from $I_1\cup I_2$ to $\R^d$ defined as follows
	\begin{equation*}
	\gamma_1\star\gamma_2(t)=
	\begin{cases}
	\gamma_1(t),&\textrm{if }t\in I_1,\\
	\gamma_2(t),&\textrm{if }t\in I_2;
	\end{cases}
	\end{equation*}
	\item the \emph{merge map} $M_{I_1,I_2}:\mathcal D_{I_1,I_2}\to\Gamma_{I_1\cup I_2}$, $(\gamma_1,\gamma_2)\mapsto \gamma_1\star\gamma_2$. We will omit the subscripts $I_1,I_2$ when clear.
\end{enumerate}

\begin{definition}\label{def:reprLAS}
	Let $T>0$, and $\boldsymbol\mu^N=\{\mu_t^N\}_{t\in[0,T]}\subseteq\cP_c(\mathbb R^d)$ be the LAS defined  in Definition \ref{def:LAS}. Denote with $I_a^b:=[a\Delta_N,b\Delta_N]$, for $a,b\in\N$, $a\le b$. We define
	\begin{enumerate}
		\item the  following measure in $\cP(\Gamma_{I_\ell^{\ell+1}})$
		\[\boldsymbol\eta^N_{I_\ell^{\ell+1}}:=\sum_{i,j} m_{i,j}^v(V[\mu^N_{\ell\Delta_N}])\delta_{\gamma_{i,j}^\ell}=\int_{T\R^d}\delta_{\gamma_{x,v}^\ell}\,d\mathcal A_N^v(V[\mu^N_{\ell\Delta_N}])(x,v),\]
		where $\gamma_{i,j}^\ell=\gamma_{x_i,v_j}^\ell$ are the solutions of the LAS characteristic system defined on $I_\ell^{\ell+1}$, thus $\gamma_{i,j}^\ell(t)=x_i+(t-\ell\Delta_N)v_j$, for $t\in I_\ell^{\ell+1}$;
		\item $\boldsymbol\eta^N_{I_0^{h+1}}:=\mu^N_{h\Delta_N}\otimes M_{I_0^h,I_h^{h+1}}\sharp(\eta^N_{I_0^h,x}\otimes\eta^N_{I_h^{h+1},x})\in\cP(\Gamma_{[0,(h+1)\Delta_N]})$, 
		defined by recursion for $h=1,\dots,N-1$, where $\eta^N_{I_0^h,x}$ and $\eta^N_{I_h^{h+1},x}$ are respectively the disintegrations of $\boldsymbol\eta^N_{I_0^h}$ and $\boldsymbol\eta^N_{I_h^{h+1}}$ w.r.t. $e_{h\Delta_N}$. That is 
\[\int_{\Gamma_{I_0^{h+1}}}\varphi(\gamma)\,d\boldsymbol\eta^N_{I_0^{h+1}}(\gamma):=\int_{\mathbb R^d}\left(\int_{e^{-1}_{h\Delta_N}(x)}\varphi(\gamma)\,d \left[M_{I_0^h,I_h^{h+1}}\sharp(\eta^N_{I_0^h,x}\otimes\eta^N_{I_h^{h+1},x})\right](\gamma)\right)\,d\mu^N_{h\Delta_N}(x),\]
for any bounded Borel function $\varphi:\Gamma_{I_0^{h+1}}\to\mathbb R$.

		\item $\boldsymbol\eta^N:=\boldsymbol\eta^N_{I_0^{N}}\in\cP(\Gamma_{[0,T]})$.
	\end{enumerate}
\end{definition}

\begin{proposition}\label{prop:reprLASok}
	Let $V$ be a PVF satisfying $(H1)$ and $(H2)$, $\mu_0\in\cP_c(\R^d)$, and $\boldsymbol\mu$ be a solution of the MDE system \eqref{MDE} obtained as uniform-in-time limit of LASs $\{\boldsymbol\mu^N\}_{N}$ for the Wasserstein metric. Let $\boldsymbol\eta^N$ be as in Definition \ref{def:reprLAS}. Then
	\begin{enumerate}
		\item\label{item1:propRA} $\boldsymbol\eta^N$ is a probabilistic representation for the LAS $\boldsymbol\mu^N$, i.e. $\mu_t^N=e_t\sharp\boldsymbol\eta^N$ for all $t\in[0,T]$;
		\item\label{item2:propRA} $\boldsymbol\eta^N\rightharpoonup^*\boldsymbol\eta$ up to subsequences, and $\boldsymbol\eta$ is a probabilistic representation for $\boldsymbol\mu$.
	\end{enumerate}
\end{proposition}

\begin{proof}[Proof of \eqref{item1:propRA}.]
	
 First we prove that $\mu_t^N= e_t\sharp\boldsymbol\eta^N_{I_\ell^{\ell+1}}$ for all $\ell=0,\dots,N-1$, and $t\in I_\ell^{\ell+1}$.
	By Definition \ref{def:LAS}, $\mu_t^N:=\sum_{i,j}m_{i,j}^v(V([\mu^N_{\ell\Delta_N}])\,\delta_{x_i+(t-\ell\Delta_N)v_j}$. Thus, for all $\varphi\in C^0_b(\R^d)$,
	\begin{align*}
	\int_{\R^d}\varphi(x)\,d\mu_t^N&=\sum_{i,j}\varphi(x_i+(t-\ell\Delta_N)v_j)\cdot m_{i,j}^v(V[\mu^N_{\ell\Delta_N}])\\
	&=\sum_{i,j}\varphi(\gamma_{i,j}^\ell(t))\cdot m_{i,j}^v(V[\mu^N_{\ell\Delta_N}])=\int_{\Gamma_{I_\ell^{\ell+1}}}\varphi(\gamma(t))\,d\boldsymbol\eta^N_{I_\ell^{\ell+1}}.
	\end{align*}
	
	Let us now conclude the proof by showing that $e_t\sharp\boldsymbol\eta^N_{I_0^{h+1}}=\mu_t^N$ for all $h=1,\dots,N-1$, and $t\in I_0^{h+1}$. Indeed, for all $\varphi\in C^0_b(\R^d)$,
	\begin{align*}
	\int\varphi(\gamma(t))\,d\boldsymbol\eta^N_{I_0^{h+1}}(\gamma)&=\int\varphi(\gamma(t))\,d\left[\mu^N_{h\Delta_N}\otimes M\sharp(\eta^N_{I_0^h,y}\otimes\eta^N_{I_{h}^{h+1},y})\right](\gamma)\\
	&=\int\varphi(\gamma(t))\,d\left[M\sharp(\eta^N_{I_0^h,y}\otimes\eta^N_{I_{h}^{h+1},y})\right](\gamma)\,d\mu^N_{h\Delta_N}(y)\\
	&= \int\varphi(\gamma_1\star\gamma_2(t))\,d\eta^N_{I_0^h,y}(\gamma_1)\,d\eta^N_{I_{h}^{h+1},y}(\gamma_2)\,d\mu^N_{h\Delta_N}(y)\\
	&=\int\varphi(\gamma_2(t))\,d\eta^N_{I_h^{h+1},y}(\gamma_2)\,d\mu^N_{h\Delta_N}(y)\\
	&=\int\varphi(\gamma_2(t))\,d\boldsymbol\eta^N_{I_h^{h+1}}(\gamma_2)=\int\varphi(x)\,d\mu^N_t(x),
	\end{align*}
	where in the fourth equality we assumed, without loss of generality, $t\in I_h^{h+1}$, otherwise we iterate the same procedure for $\boldsymbol\eta^N_{I_0^h}$. In the last two passages we used what proved before, i.e. $e_t\sharp\boldsymbol\eta^N_{I_h^{h+1}}=\mu_t^N$.
	
	\medskip
	
	\noindent\emph{Proof of \eqref{item2:propRA}.} First, let us prove that the family $\{\boldsymbol\eta^N\}_N$ is tight, thus there exists $\boldsymbol\eta\in\cP(\Gamma_{[0,T]})$ such that $\boldsymbol\eta^N\rightharpoonup^*\boldsymbol\eta$, up to a non-relabeled subsequence. We proceed in the same way as in \cite[Theorem 8.2.1]{AGS}. Indeed, we use \cite[Lemma 5.2.2]{AGS} with
	\[\boldsymbol r^1:=e_0:\gamma\mapsto \gamma(0)\in\R^d,\quad \boldsymbol r^2:\gamma\mapsto\gamma-\gamma(0)\in\Gamma_{[0,T]},\]
	and we notice that $\boldsymbol r^1\times\boldsymbol r^2$ is proper, and by the previous item we have that the family $\{\boldsymbol r^1\sharp\boldsymbol\eta^N\}_N$ is given by the first marginals $\{\mu_0^N\}$ which is tight (indeed, it narrowly converges to $\mu_0$ by assumption),
	furthermore $\beta^N:=\boldsymbol r^2\sharp\boldsymbol\eta^N\in\Gamma_{[0,T]}$ satisfy for $p>1$,
	\begin{align*}
	\int_{\Gamma_{[0,T]}}\int_0^T|\dot\gamma(t)|^p\,dt\,d\beta^N(\gamma)&=\int_{\Gamma_{[0,T]}}\int_0^T\left|\frac{d}{dt}\left(\gamma-\gamma(0)\right)(t)\right|^p\,dt\,d\boldsymbol\eta^N(\gamma)\\
&=\int_0^T\int_{\Gamma_{[0,T]}}|\dot\gamma(t)|^p\,d\boldsymbol\eta^N(\gamma)\,dt\\
&\le T\cdot \sup_{t\in[0,T]}\,\sup_{(x,v)\in\text{supp}\,\mathcal A_N^v(V[\mu^N_t])}|v|^p\\
	\end{align*}
which is uniformly bounded thanks to assumption $(H1)$ on $V$ and recalling \eqref{eq:growthLAS}.
	Hence, the tightness of the family $\beta^N$ follows by \cite[Remark 5.1.5]{AGS}, since for $p>1$ the functional $\gamma\mapsto\int_0^T |\dot\gamma|^p\,dt$ (set to $+\infty$ if $\gamma\notin AC^p([0,T];\mathbb R^d)$ or if $\gamma(0)\neq 0$) has compact sublevels in $\Gamma_{[0,T]}$. Thus, the family $\{\boldsymbol\eta^N\}_N$ is tight.
	
	By $weak^*$-convergence of $\mu^N_t$ to $\mu_t$ for all $t\in[0,T]$ and of $\boldsymbol\eta^N$ to some $\boldsymbol\eta$ up to subsequences, and since from item \eqref{item1:propRA}, $\mu_t^N=e_t\sharp\boldsymbol\eta^N$, then we immediately have that $\boldsymbol\eta$ is a probabilistic representation for $\boldsymbol\mu$, i.e. $e_t\sharp\boldsymbol\eta=\mu_t$.
By construction (see \cite[Theorem 5.1.8]{AGS}), we have that $\boldsymbol\eta$ is supported on the curves $\gamma\in\Gamma_{[0,T]}$, where $\gamma$ are the uniform limits of the LASs characteristics where $\boldsymbol\eta^N$ is supported.
\end{proof}

\section{A semi-discrete  Lagrangian scheme for MDE}\label{sec:SDS}
In this section, we first define a semi-discrete in time Lagrangian scheme for \eqref{MDE} and   compare it to the LAS scheme  in Definition \ref{def:LAS}, showing that they converge to the same limit.
 Fixed $T>0$, for $N \in \N$ we set $\Delta t^N = T/N$ and we define a  partition of $[0,T]$ by 
\begin{equation}\label{partition}
D_N=\{t_{k}^{N} = k \Delta t^N, k=0, \dots, N\}.
\end{equation}  
To simplify the notation, we omit the index $N$ in $t^N_k$ and in $\Delta t^N$ if there is no ambiguity.\\
Given $\mu_0\in\cP_c(\R^d)$ and a PVF $V$, we  set
\begin{equation}\label{DTscheme}
\left\{
\begin{array}{l}
\bar \mu^N_{t_0}:=\mu_0;\\[6pt]
\bar \mu^N_{t_{k+1}} := \int_{T\R^d} \delta_{x + v \Delta t^N}d\nu_x[\bar\mu^N_{t_k}](v)d\bar\mu^N_{t_k}(x) = \bar\mu^N_{t_k} \oplus \Delta t\cdot \nu_{x}[\bar\mu^N_{t_{k}}],
\end{array}
\right.
\end{equation}
by Bochner integration formula (see \cite{Bochner}).
Applying iteratively the previous definition, we get
\begin{equation}\label{DTscheme_bis}
\begin{split}
\bar \mu^N_{t_{k+1}}
&= \int_{\R^d}\left(\int_{\R^{(k+1)\, d}} \delta_{x^N_{k+1}} \hspace{1 mm} d\nu_{x^N_{k}}[\bar\mu^N_{t_k}](v_k^N)\dots d\nu_x[\bar\mu^N_{t_0}](v_0)\right)d\mu_0(x)
\\& = \mu_0 \oplus\left(\bigoplus_{i=0}^{k} \Delta t\cdot\nu_{x_i^N}[\bar\mu^N_{t_i}]\right),
\end{split}
\end{equation}
where $x_{k+1}^N := x + \sum_{j=0}^{k}v^N_j \Delta t^N$, for  $x \in \supp\,(\mu_0)$ and $v_j^N \in \supp\,(\nu_{x^N_{j}}[\bar\mu^N_{t_j}])$.
We extend $\boldsymbol{\bar\mu}^N$ to the interval $[0,T ] $ by setting for $t \in (t_k, t_{k+1}]$
\[
\bar \mu^N_t := \int_{T\R^d} \delta_{x + v (t - t_k)}d\nu_x[\bar\mu^N_{t_k}](v)d\bar\mu^N_{t_k}(x),
\]
and we denote $\boldsymbol{\bar\mu}^N=\{\bar\mu^N_t\}_{t\in[0,T]}$.
Due to the assumptions on $V$, we have that $\bar \mu^N_{t} \in \cP(\R^d)$ for all $t \in [0,T]$. 
Moreover, since for some $R>0$, $\supp\,(\mu_0) \subset B(0,R)$, then by $(H1)$ and arguing as in    \cite[Lemma 3.3]{Piccoli} it follows that
\begin{equation}\label{exp_bound}
\supp\,(\bar\mu^N_{t} )\subset B(0, e^{CT}(R+1)), \qquad \forall t\in [0,T].
\end{equation}

\begin{theorem}\label{DTscheme_conv}
Assume $(H1)$, $(H3)$ and
let $\mu_0 \in \cP_c(\R^d)$. Then, the   scheme   \eqref{DTscheme} converges, up to a subsequence, to a solution of \eqref{MDE}.\\
Moreover, assume that  there exists a sequence $\{N_k\}_{k \in \N}$ such that both the scheme \eqref{DTscheme} and the LAS schemes in Definition \ref{def:LAS} converge.  Then, they converge to the same solution of \eqref{MDE}.
\end{theorem}
\begin{proof}
We first show that sequence \eqref{DTscheme} is equi-Lipschitz continuous in time.  For $f \in  \text{Lip}_1(\R^d)$, by $(H1)$ and \eqref{exp_bound} we have
\begin{align*}
|\langle \bar\mu^N_{t_{k+1}} - \bar\mu^N_{t_k}, f\rangle| 
&=\left| \int_{T\R^d}\left[ f(x + v\Delta t) - f(x)\right]d\nu_{x}[\bar\mu^N_{t_k}](v)d\bar\mu^N_{t_k}(x) \right|\\
&\leq \Delta t \int_{T\R^d} |v| d\nu_{x}[\bar\mu^N_{t_k}]d\bar\mu^N_{t_k}(x) \leq  \Delta t C(1 + e^{CT}(R+1)).
\end{align*}
 It follows that
 \[
W^{\R^d}(\bar\mu^N_t, \bar\mu^N_s) \leq K |t - s|, \quad \forall t,s \in [0,T],
\]
for $K = K(R, T)>0$.
By Ascoli-Arzel\`a Theorem, the sequence $\{\boldsymbol{\bar\mu}^N\}_{N \in \N}$ admits at least a subsequence,  still denoted by $\boldsymbol{\bar\mu}^N$, which converges to a measure map $\boldsymbol{\bar\mu} \in \mathrm{Lip}_K([0,T], \cP_c(\R^d))$ such that $\bar\mu_{t=0} = \mu_0$.\\
  We now prove that $\boldsymbol{\bar\mu}$ is  a solution of \eqref{MDE}. For simplicity we index with $N$ the converging subsequence. Given $t \in (t^N_k,t^N_{k+1}]$ and $f \in C^{\infty}(\R^d) \cap \mathrm{Lip}_1(\R^d)$ such that $\| f \|_{C^2(\R^d)} \leq 1$, we have
\begin{align}\label{sol1}
\begin{split}
&\langle \bar\mu_{t} -\mu_0, f\rangle - \int_0^t \int_{T\R^d} \nabla f(x) \cdot v\,dV[\bar\mu_s](x,v)\,ds\\
&=\langle \bar\mu_{t}- \bar\mu^N_{t}  , f\rangle
+\langle \bar\mu_t^N - \mu_0, f\rangle-\int_0^t \int_{T\R^d} \nabla f(x) \cdot v\,dV[\bar\mu_s](x,v)\,ds\\
&=\langle \bar\mu_{t}- \bar\mu^N_{t}  , f\rangle
+\langle \bar\mu_{t}^N - \bar\mu^N_{t_k}, f\rangle+\\
&\qquad+ \sum_{i=0}^{k-1}\left[\langle \bar\mu^N_{t_{i+1}} - \bar\mu^N_{t_i}, f\rangle - \int_{t_i}^{t_{i+1}}\int_{T\R^d} \nabla f(x) \cdot v\,dV[\bar\mu_s](x,v)\,ds\right]+\\
&\qquad- \int_{t_{k}}^t \int_{T\R^d} \nabla f(x) \cdot v\,dV[\bar\mu_s](x,v)\,ds
\end{split}
\end{align}
Recalling that $\bar \mu_{t_{i+1}}^N = \int_{T\R^d} \delta_{x + v \Delta t}d\nu_x[\bar\mu_{t_i}^N]d\bar\mu_{t_i}^N(x)$, 
we have
\begin{align*}
&\langle \bar\mu^N_{t_{i+1}} - \bar\mu^N_{t_i}, f\rangle - \int_{t_i}^{t_{i+1}}\int_{T\R^d} \nabla f(x) \cdot v\,dV[\bar\mu_s](x,v)\,ds\\
&= \int_{T\R^d} \left[ f(x+v \Delta t)  - f(x)\right]dV[\bar\mu^N_{t_i}](x,v)- \int_{t_i}^{t_{i+1}}\int_{T\R^d}\nabla f(x)\cdot v\, dV[\bar\mu_s](x,v) ds\\
&= \int_{t_i}^{t_{i+1}}\int_{T\R^d} \frac{d}{ds}f(x + (s - t_i)v) \,dV[\bar\mu^N_{t_i}](x,v)ds - \int_{t_i}^{t_{i+1}}\int_{T\R^d}\nabla f(x)\cdot v \,dV[\bar\mu_s](x,v) ds\\
&= \int_{t_i}^{t_{i+1}}\int_{T\R^d} \nabla f(x + (s - t_i)v)\cdot v \,dV[\bar\mu^N_{t_i}](x,v)ds - \int_{t_i}^{t_{i+1}}\int_{T\R^d}\nabla f(x)\cdot v \,dV[\bar\mu_s](x,v) ds\\
&= \int_{t_i}^{t_{i+1}}\int_{T\R^d} (\nabla f(x + (s - t_i)v) - \nabla f(x))\cdot v \,dV[\bar\mu^N_{t_i}](x,v)ds +\\
&\qquad+ \int_{t_i}^{t_{i+1}}\int_{T\R^d}\nabla f(x)\cdot v \,d\left( V[\bar\mu^N_{t_i}]- V[\bar\mu_s]\right)(x,v) ds\\
&= I_{1,i} + I_{2,i}.
\end{align*}
We now  estimate $I_{1,i}$ and $I_{2,i}$. By  $(H1)$ and \eqref{exp_bound}, we have
\begin{align*}
|I_{1,i}|
&\leq \int_{t_i}^{t_{i+1}} \int_{T\R^d}|\nabla f(x + (s - t_i)v) - \nabla f(x)|\cdot |v| \,dV[\bar\mu^N_{t_i}](x,v)ds \\
&\leq \| f\|_{C^2(\R^d)}\int_{t_i}^{t_{i+1}}\int_{T\R^d} |s - t_i||v|^2 d V[\bar \mu^N_{t_i}](x,v)ds\\
&\leq \| f\|_{C^2(\R^d)} C^2(1 + e^{CT}(R+1))^2 \Delta t^2.
\end{align*}
By the Kantorovich-Rubinstein duality, $(H3)$ and the triangular inequality, we get
\begin{align*}
I_{2,i}
&\leq \int_{t_i}^{t_{i+1}} W^{T\R^d}(V[\bar\mu^N_{t_i}], V[\bar\mu_s])ds \leq \int_{t_i}^{t_{i+1}} L\cdot W^{\R^d}(\bar\mu^N_{t_i}, \bar\mu_s)ds\\
&\leq \int_{t_i}^{t_{i+1}} L (W^{\R^d}(\bar\mu^N_{t_i}, \bar\mu^N_s) + W^{\R^d}(\bar\mu^N_s,\bar \mu_s))ds\leq \int_{t_i}^{t_{i+1}}L (K(s - t_i) + W^{\R^d}(\bar\mu^N_s,\bar \mu_s))ds\\
&\leq LK\left(\Delta t^2 + \int_{t_i}^{t_{i+1}}W^{\R^d}(\bar\mu^N_s, \bar\mu_s) ds\right).
\end{align*}
Replacing the previous estimates in \eqref{sol1}, we get
\begin{align*}
&\langle \bar\mu_{t} -\mu_0, f\rangle - \int_0^t \int_{T\R^d} \nabla f(x) \cdot v\,dV[\bar\mu_s](x,v)\,ds\\
& \le\langle \bar\mu_{t}- \bar\mu^N_{t}  , f\rangle +\sum_{i=0}^{k-1} (I_{1,i} + I_{2,i})+ \langle \bar \mu^N_t - \bar\mu^N_{t_k}, f\rangle - \int_{t_k}^t \int_{T\R^d} \nabla f(x) \cdot v\,dV[\bar\mu_s](x,v)\,ds\\
&\leq
\langle \bar\mu_{t}- \bar\mu^N_{t}  , f\rangle+ LK \int_0^{t}W^{\R^d}(\bar\mu^N_s, \bar\mu_s)ds + 2K' \,T\,\Delta t,
\end{align*}
where $K' = \max\{LK, C^2(1 + e^{CT}(R+1))^2\}$. Passing to the limit for $N\to \infty$ in the previous inequality and  recalling that  $W^{\R^d}(\bar\mu^N_t, \bar\mu_t) \to 0$, we finally get that
\[  \langle \bar\mu_{t} -\mu_0, f\rangle - \int_0^t \int_{T\R^d} \nabla f(x) \cdot v\,dV[\bar\mu_s](x,v)\,ds=0\]
for any $f$ and therefore $\boldsymbol{\bar\mu}$ is a solution of \eqref{MDE}.\par 
Let us now prove the second part of the theorem. Let $\{\boldsymbol\mu^N\}$ be  a convergent (sub-) sequence generated by the  LAS scheme  in  Definition \ref{def:LAS}, and  $\{\boldsymbol{\bar\mu}^N\}$ be a convergent  one generated by the scheme \eqref{DTscheme}. Let us denote by  $\boldsymbol\mu, \boldsymbol{\bar\mu}$ the corresponding  limits. Then
\begin{align*}
W^{\R^d}(\mu_t, \bar \mu_t) \leq W^{\R^d}(\mu_t, \mu^N_t) + W^{\R^d}(\bar\mu^N_t, \bar\mu_t) +W^{\R^d}(\mu^N_t, \bar\mu^N_t).
\end{align*}
Since the first two terms on the right-hand side of the last inequality converge to $0$ for $N \to +\infty,$ we have to study only  the convergence of the last term. Let $f \in \mathrm{Lip}_1(\R^d)$, $t\in(t_k,t_{k+1}]$, then
\begin{align}
\begin{split}\label{eq:convDSLAS}
\langle \mu_t^N - \bar\mu^N_t, f\rangle &=\int_{T\R^d} f(x + (t - t_k) v)\,d(V[\mu_{t_k}^N] - V[\bar\mu_{t_k}^N])\\&+  \int_{T\R^d}f(x + (t - t_k)v)\,d\left(\mathcal{A}^v_N(V[\mu_{t_k}^N]) - V[\mu_{t_k}^N]\right).
\end{split}
\end{align}
Notice that this computation holds thanks to the common time-grid shared by the two schemes.
For the
first term, we have
\begin{align*}
&\int_{T\R^d} f(x + (t - t_k) v)\,d(V[\mu_{t_k}^N] - V[\bar\mu_{t_k}^N]) = \int_{T\R^d} f(x+w )\, d( V^{\Delta t^N}[\mu_{t_k}^N] - V^{\Delta t^N}[\bar\mu_{t_k}^N]),
\end{align*}
where we have denoted $dV^{\Delta t^N}[\eta] = d((t-t_k) \cdot \nu_x[\eta])d\eta$, referring to the notation in \eqref{eq:prodDef}.\\
Then, we can observe that the map $\psi:(x,w) \to x+w$ belongs to $\mathrm{Lip}_1(T\R^d, \R^d)$. Since $f \in \mathrm{Lip}_1(\R^d, \R)$, we have $f\circ \psi \in \mathrm{Lip}_1(T\R^d, \R)$. Then, from the previous inequality and the Kantorovich-Rubinstein duality it follows that
\begin{align*}
&\int_{T\R^d} f(x + (t - t_k) v)\,d(V[\mu_{t_k}^N] - V[\bar\mu_{t_k}^N])\leq W^{T\R^d}(V^{\Delta t^N}[\mu_{t_k}^N], V^{\Delta t^N}[\bar\mu_{t_k}^N])\\
 &\leq \Delta t^N \mathcal{W}(V[\mu_{t_k}^N] , V[\bar\mu_{t_k}^N]) + W^{\R^d}(\mu_{t_k}^N,\bar\mu_{t_k}^N)
\leq (L\Delta t^N  +1)W^{\R^d}(\mu_{t_k}^N,\bar\mu_{t_k}^N),
\end{align*}
where the last inequality is a consequence of $(H3)$. For the second term in \eqref{eq:convDSLAS}, by the same argument and \eqref{eq:estimLAS}, we found it is bounded by $\frac{1}{N^2}$. Then
\begin{align*}
\langle \mu_t^N - \bar\mu^N_t, f\rangle
&\leq \frac{1}{N^2} + (1 + L \Delta t^N)W^{\R^d}(\mu^N_{t_k}, \bar\mu^N_{t_k})\\
&\leq  \left(1 + L \Delta t^N\right)^{k+1} o\left(\frac{1}{N}\right) + \frac{\sum_{l=0}^k \left(1 + L\Delta t^N\right)^l}{N^2}\\
&\leq e^{\frac{L T(k+1)}{N}} \cdot o\left(\frac{1}{N}\right) + \frac{e^{\frac{L T(k+1)}{N}} - 1}{NLT}\\
\end{align*}
and therefore the two schemes converge to the same limit, up to subsequences.
\end{proof}

Before giving a further consideration coming as a consequence of the previous theorem, we recall the following uniqueness result proved in \cite[Theorem 5.2]{Piccoli}.
\begin{theorem}\label{thmPic:semig}
Let $V$ be a PVF satisfying $(H1)$ and $(H3)$. Assume that, for every $\mu_0$ obtained as convex combination of Dirac deltas, the sequence of LASs converges to a unique limit. Then there exists a unique Lipschitz semigroup whose trajectories are limits of LASs.
\end{theorem}

Then, as a corollary of Theorem \ref{DTscheme_conv}, we are able to get a uniqueness result for a Lipschitz semigroup of MDEs obtained as limit of the semi-discrete Lagrangian scheme, as follows.
\begin{corollary}\label{cor:uniqSemiLagr}
Under the assumptions of Theorem \ref{thmPic:semig}, there exists a unique Lipschitz semigroup generated by the semi-discrete Lagrangian scheme \eqref{DTscheme} and it coincides with that generated by LASs.
\end{corollary}

\section{A mean velocity scheme for MDE}\label{sec:MVP}
In this section, inspired by the discussion developed in Section \ref{sec:SP}, we provide another  approximation scheme for the problem \eqref{MDE}.
By the result proved in Proposition \ref{prop:equivMDE-CE}, this scheme is analogous to the semi-discrete in time Lagrangian scheme studied in \cite[Scheme 1]{PR} for a nonlocal continuity equation, indeed it is based on the choice of a barycentric velocity field.
As our Lagrangian scheme in Section \ref{sec:SDS}, also this scheme is semi-discrete in time
but, due to a different choice of the velocity field, it may converge to  a different solution of the MDE
(see also Remark \ref{lack_uni} for a discussion about general lack of uniqueness). This fact will be exploited in Section \ref{sec:ex} with clarifying examples.

\medskip

We define $\Delta t^{N}$ and $t^N_k$ as in Section \ref{sec:SDS}. Given $\mu_0\in\cP_c(\R^d)$ and a PVF $V$ satisfying $(H1)$-$(H2)$, the new approximation scheme is given iteratively by
\begin{equation}\label{MVP}
\left\{
\begin{array}{ll}
\hat\mu^N_{t=0} = \mu_0;\\[6pt]
\bar v_{t_j}(x) := \int_{\R^d} v d\nu_x[\hat\mu^N_{t_j}](v),&\\
\hat\mu^N_{t_{j+1}} = \hat\mu^N_{t_j} \oplus \Delta t^N\cdot \delta_{\bar v_{t_j}(x)},
\end{array}
\right.
\end{equation}
for $ j \in \{0, \dots, N-1\}$.
The  scheme \eqref{MVP} transports the measure distribution $\hat\mu^N_{t_j}$ by a velocity field obtained as the barycenter of the velocity measure $\nu_x$ at $\hat\mu^N_{t_j}$. 

The following is another existence result for solutions of  \eqref{MDE}, proved through the mean velocity approximation scheme.
\begin{theorem}\label{thm:convMV}
Assume $(H1)$-$(H2)$ and let $\mu_0 \in \cP_c(\R^d)$. Then, the   scheme   \eqref{MVP} converges, 
	up to a subsequence, to a solution of \eqref{MDE}.
\end{theorem}
\begin{proof}
	We first prove that, given $R>0$ such that $\mathrm{supp}(\mu_0) \subset B(0,R)$, there exists $K=K(R,T)>0$ such that for $N$ sufficiently large
	\begin{equation}\label{cs}
	\supp (\hat\mu^N_{t_j}) \subset B(0,K), \quad \forall j \in \{ 0, \dots, N\} .
	\end{equation}
	Indeed
	\begin{align*}
	&\sup_{\mathrm{supp}(\hat\mu^N_{t_{j+1}})}|x|= \sup_{\mathrm{supp}(\hat\mu^N_{t_j})}|x + \Delta t \bar v_{t_j}(x)|\leq \sup_{\mathrm{supp}(\hat\mu^N_{t_j})}|x| + \Delta t \sup_{\mathrm{supp}(V[\hat\mu^N_{t_j}])}|v|\\
	&\leq \Delta t C + (1 + \Delta t C)\sup_{\mathrm{supp}(\hat\mu^N_{t_j})}|x|\leq \dots\leq \Delta t C \sum_{k=0}^{j} (1 + \Delta t C)^k + (1 + \Delta t C)^{j+1}\,R\\
	&\leq e^{Cj\Delta t} - 1+Re^{TC}\le e^{TC}(1+R)-1,
	\end{align*}
	where we used $(H1)$.
	We now prove the equicontinuity in time of the scheme: $t,s \in [0,T]$ and $N>>1$ such that $t - s>\Delta t^N$; then there exists $\{t^N_j\}=\{t_j\}$ such that $s < t_{j} < \hdots t_{j+k} < t.$
	Let $f \in \mathrm{Lip}_1(\R^d)$, then
	\begin{align*}\langle \hat\mu^N_t - \hat\mu^N_s , f \rangle &=  \langle \hat\mu^N_t - \hat\mu^N_{t_{j + k}}, f\rangle + \sum_{i = 1}^{k} \langle \hat\mu^N_{t_{j+i}} - \hat\mu^N_{t_{j +i-1}}, f\rangle + \langle \hat\mu^N_{t_{j}} - \hat\mu^N_s, f\rangle.
	\end{align*} 
	By construction,
	\begin{align*}
	\langle \hat\mu^N_{t_{j+1}}- \hat\mu^N_{t_j}, f\rangle &= \int_{T\R^d} \left(f(x + \Delta t^N \bar v_j) - f(x)\right)d\hat\mu^N_{t_j}\\
	&\leq \Delta t^N \int_{\R^d}|\bar v_j(x)|d\hat\mu^N_{t_j}(x)\leq \Delta t^N \int_{T\R^d} |v| dV[\hat\mu^N_{t_j}].
	\end{align*}
	Hence by $(H1)$ and  the equi-boundedness of supports, it follows 
	\begin{align*}
	\langle \hat\mu^N_{t_{j+1}}-\hat\mu^N_{t_j}, f\rangle & \leq \Delta t^N \int_{\R^d} C(1 + |x|) \,d\hat\mu^N_{t_j}(x)\leq \Delta t^N C(1 + K).
	\end{align*}
	Analogously $\langle \hat\mu^N_t - \hat\mu^N_{t_{j+k}}, f\rangle \leq |t - t_{j+k}| C(1 + K)$ and $\langle \hat\mu^N_{t_j}- \hat\mu^N_{s}, f\rangle \leq |t_j - s| C(1 + K)$.\\ Hence, taking the supremum for $f \in \mathrm{Lip}_1(\R^d),$ we have
	$$W^{\R^d}(\hat\mu^N_t, \hat\mu^N_s) \leq |t - s| C(1 + K).$$
	Since the support of $\hat\mu^N_{t_j}$  is bounded, uniformly in $N$, it immediately follows that the sequence $\{\boldsymbol{\hat\mu}^N\}_{N\in\N}$ have bounded first and second momentum  and therefore there exists    $\boldsymbol\mu \in C([0,T];\cP_c(\R^d))$ such that, up to a subsequence,
\begin{equation}\label{eq:Wconv}
\sup_{t \in [0,T]}W^{\R^d}(\hat\mu^N_t, \mu_t) \to 0, \textrm{ for } N \to +\infty.
\end{equation}
	We now prove that $\boldsymbol\mu$ is a solution of \eqref{MDE}. Given $f \in C^{\infty}_c(\R^d)$ such that $\|f\|_{C^2(\R^d)}\le 1$, we   rewrite
	\begin{equation}\label{proof:mss1}
	\begin{split}
	&	\langle \hat\mu^N_{t_k } - \hat\mu^N_0,f\rangle = \sum_{j=1}^{k} \langle \hat\mu^N_{t_j} - \hat\mu^N_{j -1}, f \rangle,\\
	&	=\sum_{j=1}^{k} \int_{\R^d}\left( f(x + \Delta t \bar v_{t_{j-1}}(x)) - f(x) \right) d\hat\mu^N_{t_{j-1}}\\
	&=\sum_{j=1}^{k}\int_{\R^d}\int_{t_{j-1}}^{{t_j}} \frac{d}{ds} f(x + (s - t_{j-1})\bar v_{t_{j-1}}(x)) ds d\hat\mu^N_{t_{j-1}}(x)\\
	&=\sum_{j=1}^{k}\int_{t_{j-1}}^{t_j} \int_{\R^d} \bar v_{t_{j-1}}(x)\cdot\nabla f(x + (s - t_{j-1})\bar v_{t_{j-1}}(x))d\hat\mu^N_{t_{j-1}}(x)ds.
	\end{split}
	\end{equation}
	We estimate 
	\begin{align*}
	\begin{split}
	&\left|\int_{t_{j-1}}^{t_j}  \left(\int_{\R^d}\bar v_{t_{j-1}}(x)\cdot\nabla f(x + (s - t_{j-1})\bar v_{t_{j-1}}(x))d\hat\mu^N_{t_{j-1}}(x) - \int_{T\R^d} v\cdot\nabla f(x)dV[\mu_s]\right)ds\right|\\
	&=	\left|\int_{t_{j-1}}^{t_j} \left( \int_{T\R^d} v\cdot \nabla f(x + (s - t_{j-1})\bar v_{t_{j-1}}(x)) dV[\hat\mu^N_{t_{j-1}}] - \int_{T\R^d} v\cdot\nabla f(x)dV[\mu_s] \right)ds\right|\\
	&\leq	\int_{t_{j-1}}^{t_j}  \left| \int_{T\R^d} v\cdot \nabla \left(f(x + (s - t_{j-1})\bar v_{t_{j-1}}(x)) - f(x)\right) dV[\hat\mu^N_{t_{j-1}}]\right|ds \\
	&+	\int_{t_{j-1}}^{t_j} \left| \int_{T\R^d} v\cdot \nabla f(x)d \left(V[\hat\mu^N_{t_{j-1}}] -  V[\mu_s]\right)\right|ds\\
	&\leq	\int_{t_{j-1}}^{t_j} (s - t_{j-1}) \int_{\R^d}\|D^2 f\|_{C^2} |\bar v_{t_{j-1}}(x)|^2 d\hat\mu^N_{t_{j-1}}(x) + \int_{t_{j-1}}^{t_j} W^{T\R^d}(V[\hat\mu^N_{t_{j-1}}],V[\mu_s])ds\\
	&\leq	\int_{t_{j-1}}^{t_j} (s - t_{j-1}) \int_{T\R^d} |v|^2 dV[\hat\mu^N_{t_{j-1}}](x)ds + \Delta t \sup_{t \in [ t_{j-1},t_{j}]} W^{T\R^d}(V[\hat\mu^N_{t_{j-1}}],V[\mu_t]).
	\end{split}
	\end{align*}
	Therefore, by  $(H1)$ and \eqref{cs}, we have 
	\begin{align*}
	\begin{split}
	&\left| \sum_{j=1}^{k}\int_{t_{j-1}}^{t_j} \int_{\R^d} \bar v_{t_{j-1}}(x)\nabla f(x + (s - t_{j-1})\bar v_{t_{j-1}}(x))d\hat\mu^N_{t_{j-1}}(x)ds-\int_0^T
	\int_{T\R^d} v\cdot\nabla f(x)dV[\mu_s](x,v)ds\right|\\
	&	\le \sum_{j=1}^{k}\left(\int_{t_{j-1}}^{t_j} (s - t_{j-1}) \int_{T\R^d} |v|^2 dV[\hat\mu^N_{t_{j-1}}](x)ds + \Delta t \sup_{t \in [t_{j-1},t_{j}]} W^{T\R^d}(V[\hat\mu^N_{t_{j-1}}],V[\mu_t])\right)\\
	&\leq	\sum_{j=1}^{k} \frac{\Delta t^2}{2} C' + T \sup_{ 1\leq j \leq k} \sup_{t \in [t_{j-1},t_{j}]}W^{T\R^d}(V[\hat\mu^N_{t_{j-1}}],V[\mu_t])\\
	&\leq	\Delta t T C' + T \sup_{ 1\leq j \leq k} \sup_{t \in [t_{j-1},t_{j}]}W^{T\R^d}(V[\hat\mu^N_{t_{j-1}}],V[\mu_t]).\\
	\end{split}
	\end{align*}
	By $(H2)$, $V$ is uniformly continuous on $B(0,K)$, hence we conclude that the right hand-side vanishes as $N\to+\infty$, thanks to \eqref{cs} and \eqref{eq:Wconv}.
	If $t^N_k\to t$ for $N\to \infty$, since the term on the left side in \eqref{proof:mss1}  converges to $\langle \mu_t - \mu_0, f\rangle$ by construction,  the previous estimate implies that
	$\mu$ is a weak solution to     \eqref{MDE}.
\end{proof}

\begin{remark}
For sake of completeness, similarly to Definition \ref{def:reprLAS}, we can provide an explicit formula to construct a probabilistic representation for the scheme introduced in Section \ref{sec:SDS} and for the mean velocity one, as follows.
Let $I_a^b:=[a\Delta t^N,b\Delta t^N]$, for $a,b\in\N$, $a\le b$, with $\Delta t^N$ and $t_k$ as in Section \ref{sec:SDS}. Denote respectively with $\boldsymbol{\bar\mu}^N=\{\bar\mu_t^N\}_{t\in[0,T]}$ and $\boldsymbol{\hat\mu}^N=\{\hat\mu_t^N\}_{t\in[0,T]}$ the schemes defined in \eqref{DTscheme} and \eqref{MVP}. For $k=0,\dots N-1$, we define
\begin{itemize}
\item[(A)] $\boldsymbol{\bar\eta}^N_{I_k^{k+1}}:=\bar\mu^N_{t_k}\oplus(\cdot -t_k)\cdot \nu_x[\bar\mu^N_{t_k}]\in\cP(\Gamma_{I_k^{k+1}})$, i.e.,
\[\boldsymbol{\bar\eta}^N_{I_k^{k+1}}=\int_{\R^d\times\R^d}\delta_{x+(\cdot-t_k)v}\,d\nu_x[\bar\mu^N_{t_k}](v)\,d\bar\mu^N_{t_k}(x)=\int_{T\R^d}\delta_{x+(\cdot-t_k)v}\,dV[\bar\mu^N_{t_k}](x,v);\]
\item[(B)] $\boldsymbol{\hat\eta}^N_{I_k^{k+1}}:=\hat\mu^N_{t_k}\oplus(\cdot -t_k)\cdot \delta_{\bar v_{t_k}(x)}\in\cP(\Gamma_{I_k^{k+1}})$, i.e.,
$\boldsymbol{\hat\eta}^N_{I_k^{k+1}}=\displaystyle\int_{\R^d}\delta_{x+(\cdot-t_k)\bar v_{t_k}(x)}\,d\hat\mu^N_{t_k}(x)$.
\end{itemize}
Now, we can build $\boldsymbol{\bar\eta}^N$ and $\boldsymbol{\hat\eta}^N$ by applying items $(2-3)$ of Definition \ref{def:reprLAS} and replacing item $(1)$ respectively with (A) and (B).

Following the same line as in Proposition \ref{prop:reprLASok}, we can prove that an analoguos result holds also for the semi-discrete in time Lagrangian scheme (or the mean velocity one) by replacing the LAS scheme $\boldsymbol\mu^N$ with the scheme $\boldsymbol{\bar\mu}^N$ (or $\boldsymbol{\hat\mu}^N$) and using the representation $\boldsymbol{\bar\eta}^N$ (or $\boldsymbol{\hat\eta}^N$) just provided.
\end{remark}

\section{Examples}\label{sec:ex}
In this section we present some examples in $d=1$ aimed at clarifying the work of the various proposed schemes, in particular we show that the LAS scheme and the mean velocity one in \eqref{MVP} may converge to different solutions of a MDE. We also exploit the result stated in Proposition \ref{prop:equivMDE-CE}.
For simplicity of computations and without loss of generality, let us set $\Delta_N=\Delta t^N=1/N$ as a time-step size for all the schemes. We denote with $\mathcal L=\mathcal L^1$ the $1$-dimensional Lebesgue measure.

\begin{example}[Splitting particle]\label{ex1} 

For every $\mu\in\mathscr P_c(\mathbb R)$ define:
\[
B(\mu)=\sup \left\{x:\mu(]-\infty,x])\leq \frac12\right\}.
\]
Set $\eta(\mu)=\mu(]-\infty,B(\mu)]) - \frac12$ so 
$\mu(\{B(\mu)\})=\eta(\mu)+\frac12-\mu(]-\infty,B(\mu)[)$.
We define $V[\mu]=\mu\otimes \nu_x[\mu]$, with
\begin{equation}\label{eq:PVF-Bar}
\nu_x[\mu]=\left\{
\begin{array}{ll}
\delta_{-1} & \textrm{if}\ x<B(\mu)\\
\delta_{1} & \textrm{if}\ x>B(\mu)\\
\frac{1}{\mu(\{B(\mu)\})}
\left(\eta\delta_{1}+
\left(\frac12-\mu(]-\infty,B(\mu)[)\right)\delta_{-1}\right) & \textrm{if}\ x=B(\mu),
\mu(\{B(\mu)\})>0.
\end{array}
\right.
\end{equation}
The solution obtained as limit of LASs, satisfies:
\[
\mu_t(A) = \mu_0((A\cap ]-\infty,B(\mu_0)-t[)+t)
+ \mu_0((A\cap ]B(\mu_0)+t,+\infty[)-t)
\]
\[
+\frac{1}{\mu_0(\{B(\mu_0)\})}
\left( \eta \delta_{B(\mu_0)+t}(A)+ (\frac12-\mu_0(]-\infty,B(\mu_0)[))\delta_{B(\mu_0)-t}(A)\right).
\]
In particular:
\begin{itemize}
\item[i)] The solution with $\mu_0=\delta_{x_0}$
is given by $\mu_t=\frac12 \delta_{x_0+t}+\frac12 \delta_{x_0-t}$, as illustrated in \textbf{Figure \ref{fig1}};
\item[ii)] The solution with $\mu_0=\chi_{[a,b]}\,\lambda$
(where $\chi$ is the characteristic function and $\lambda=\frac{1}{b-a}\mathcal L\llcorner_{[a,b]}$) 
is given by $\mu_t=\chi_{[a-t,\frac{a+b}{2}-t]}\lambda +
\chi_{[\frac{a+b}{2}+t,b+t]}\lambda$.
\end{itemize}
The same behavior is valid for the scheme \eqref{DTscheme} (see Theorem \ref{DTscheme_conv}). Moreover, it can be verified that the stationary solution $\{\delta_{x_0}\}_{t}$ is  the unique limit of the mean velocity scheme \eqref{MVP} when $\mu_0=\delta_{x_0}$, while this scheme has the same behavior of the LASs one when $\mu_0= \chi_{[a,b]}\,\lambda$. Hence the limit solution depends, in general, on the given approximation scheme.\\
Let us read these results in light of Proposition \ref{prop:equivMDE-CE} and Remark \ref{lack_uni}, in the easiest case when we start from $\mu_0=\delta_{x_0}$. First, we notice that the maps $x\mapsto\nu_x[\mu]$ and thus also
\begin{equation}\label{eq:wex}
x\mapsto w[\mu](x):=\int_{\pi_1^{-1}(x)}v\,d\nu_x[\mu](v)
\end{equation}
are not even continuous, thus uniqueness of solutions to the continuity equation $\partial_t\mu_t+\mathrm{div}(w[\mu_t]\mu_t)=0$, or equivalently to the MDE $\dot\mu_t=V[\mu_t]$, with $\mu_{|t=0}=\mu_0$, is not guaranteed. We can verify that both the curves $\boldsymbol\mu^1=\{\delta_{x_0}\}_t$ and $\boldsymbol\mu^2=\{\mu_t\}_t$, with $\mu_t=\frac12 \delta_{x_0+t}+\frac12 \delta_{x_0-t}$, selected by the schemes \eqref{DTscheme} and \eqref{MVP} are in fact solutions. Indeed, along $\boldsymbol\mu^1$, we have $B(\delta_{x_0})=x_0$ and $\eta(\delta_{x_0})=\frac{1}{2}$, thus $\nu_{x_0}[\delta_{x_0}]=\frac{1}{2}\delta_1+\frac{1}{2}\delta_{-1}$ and $w[\delta_{x_0}](x_0)=0$. Hence straightforwardly, \eqref{eq:contNL} with $w$ as in \eqref{eq:wex} and so the MDE are satisfied for $\boldsymbol\mu^1$.

Concerning $\boldsymbol\mu^2$, for $t> 0$ we have $B(\mu_t)=x_0+t$, $\eta(\mu_t)=\frac{1}{2}$, thus $\nu_{x_0-t}[\mu_t]=\delta_{-1}$ and $\nu_{x_0+t}[\mu_t]=\delta_{1}$. Hence,
\begin{align*}
\int_{T\mathbb R^d}(\nabla f(x)\cdot v)\,dV[\mu_t](x,v)&=\int_{\mathbb R^d}\nabla f(x) \cdot \left[\int_{\pi_1^{-1}(x)}v\,d\nu_x[\mu_t](v)\right]\,d\mu_t(x)\\
&=\frac{1}{2}\nabla f(x_0+t)-\frac{1}{2}\nabla f(x_0-t)\\
&=\frac{d}{dt}\int_{\mathbb R^d}f(x)\,d\mu_t(x),
\end{align*}
thus also $\boldsymbol\mu^2$ is a solution to the same MDE (or equivalently to \eqref{eq:contNL} with $w$ as in \eqref{eq:wex}).
\begin{figure}[h!]
\centering
\includegraphics[scale=0.65]{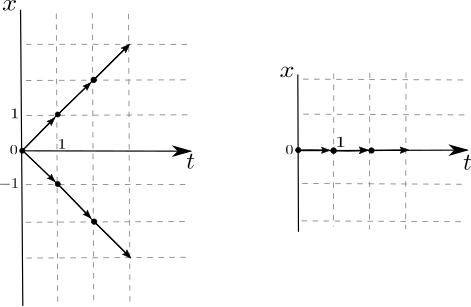}
\centering
\caption{LAS and semi-discrete Lagrangian schemes on the left for $N=1$. Mean-velocity scheme on the right. }
\label{fig1}
\end{figure}
\end{example}

\begin{example}\label{ex2} 
Let $V:\mathscr P(\mathbb R)\to\mathscr P(T\mathbb R)$ a PVF defined by $V[\mu]:=\mu\otimes\omega$, where $\omega:=\frac{1}{2}\left(\delta_{-1}+\delta_1\right)$, and let $\mu_0=\delta_0$. Then, both the LAS and \eqref{DTscheme} schemes give a binomial distribution at every time (see \textbf{Figure \ref{fig2}}) while, as in the previous example, the mean velocity scheme remains stationary. However by the Law of Large Numbers, as $N\to+\infty$ all the three schemes univocally converge to the constant solution $\boldsymbol\mu=\{\delta_0\}_{t\in[0,T]}$. We refer to \cite[Proposition 7.1]{Piccoli} for a formal proof.\\
This limit behavior is in line with Proposition \ref{prop:equivMDE-CE}, indeed the MDE under consideration is equivalent to the continuity equation driven by
\[w[\mu](x)=\int_{\pi_1^{-1}(x)}v\,d\omega(v)=0,\]
for $\mu$-a.e. $x\in\mathbb R^d$. By standard theory, the continuity equation driven by such a vector field, $\partial_t\mu_t+\mathrm{div}(w\mu_t)=0$, admits a unique solution. Hence, also the MDE as a unique solution, that is the stationary one. 
\begin{figure}[h!]
\centering
\includegraphics[scale=0.65]{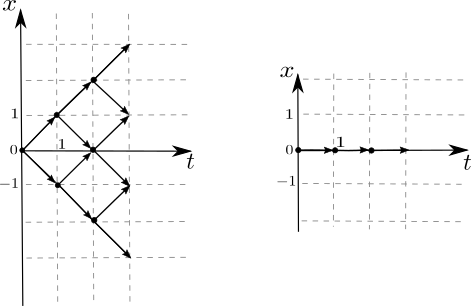}
\centering
\caption{LAS and semi-discrete Lagrangian schemes on the left. Mean-velocity scheme on the right. }
\label{fig2}
\end{figure}
\end{example}

\begin{example} 
Let $V:\mathscr P(\mathbb R)\to\mathscr P(T\mathbb R)$ a PVF defined by $V[\mu]:=\mu\otimes\omega$, where $\omega:=\frac{1}{2}\chi_{[-1,1]}\mathcal L=\frac{1}{2}\mathcal L\llcorner_{[-1,1]}$. Let $\mu_0=\delta_0$. Considering the LAS scheme, as illustrated in \textbf{Figure \ref{fig3.1}}, we notice that for $N=1$, the points $v_j$ in the discretized space of velocities such that $m_{ij}^v(V[\mu])\neq0$ are $v_0=-1$ and $v_1=0$, with equal weight. For $N=2$, we get $v_0=-1$, $v_1=-1/2$, $v_2=0$ and $v_3=1/2$, hence we start to give mass also to positive $x\in\mathbb R$, thus obtaining $\mu^{2}_{|t=1/2}=\frac{1}{4}\sum_{i=0}^3\delta_{(-1/2+i/4)}$ and $\mu^{2}_{|t=1}=\sum_{i=-2}^4\frac{1}{16}(4-|i-1|)\delta_{(-1/2+i/4)}$. 
\begin{figure}[h!]
\centering
\includegraphics[scale=0.65]{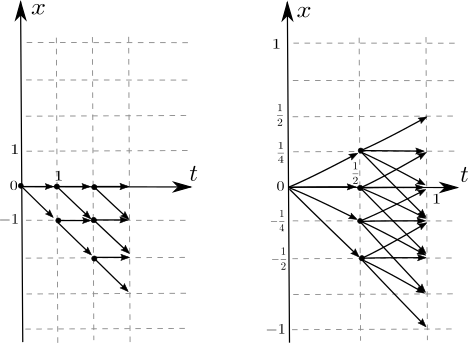}
\centering
\caption{LAS scheme: for N = 1 (left) and N=2 (right). }
\label{fig3.1}
\end{figure}
\\Coming to the semi-discrete Lagrangian scheme \eqref{DTscheme}, at the first time-step we get the uniform distribution on $[-1,1]$, while afterwards we obtain a normal distribution on $[-t,t]$ (see \textbf{Figure \ref{fig3.2}}). Reasoning in the same way as in the previous example, by the Law of Large Numbers, the LAS scheme and so also the semidiscrete Lagrangian one converge to the constant solution as $N\to\infty$ (see \cite[Proposition 7.1]{Piccoli}). Trivially, the mean-velocity scheme shares the same behavior.\\
We can get this conlcusion also through the very same discussion analyzed in Example \ref{ex2}, indeed by Proposition \ref{prop:equivMDE-CE}, we get uniqueness of solutions for this MDE, being equivalent to a continuity equation driven by $w=0$.
\begin{figure}[h]
\centering
\includegraphics[scale=0.65]{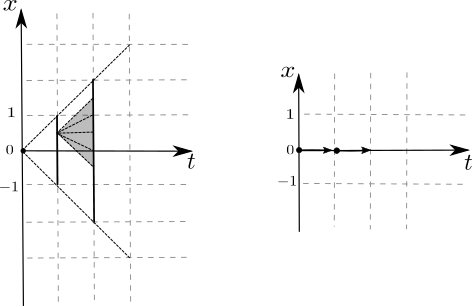}
\centering
\caption{Semi-discrete Lagrangian scheme on the left. Mean-velocity scheme on the right.}
\label{fig3.2}
\end{figure}
\end{example}

\begin{example} 
Let $V:\mathscr P(\mathbb R)\to\mathscr P(T\mathbb R)$ a PVF defined by $V[\mu]:=\mu\otimes\delta_{v(x)}$, where $v(x):=2 \sqrt{|x|}$, and $\mu_{|t=0}=\delta_{-1}$.
Recalling Definition \ref{def:sol}, by the atomic nature of the PVF $V$ over the fibers $T_x\R$, we deduce that the set of solutions to the MDE coincides with the set of distributional solutions of the continuity equation driven by the vector field $v(\cdot)$.
We can thus use the classical Superposition Principle \cite[Theorem 8.2.1]{AGS} to build the trajectories $\boldsymbol\mu$ by considering the integral solutions of the underlying ODE $\dot x(t)=2\sqrt{|x(t)|}$, with initial condition $x(0)=-1$.
\begin{figure}[h!]
\centering
\includegraphics[scale=0.65]{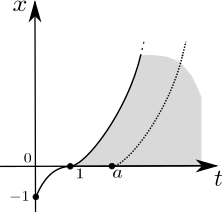}
\centering
\caption{Peano's brush referred to Example 4. }
\label{fig4.1}
\end{figure}
By classical theory we know that this system admits infinite solutions (called Peano's brush), such as the trivial one $x_\infty(t)=-(t-1)^2$ for $t\le 1$, $x_\infty(t)=0$ for $t\ge1$, but also the trajectories given by
\begin{equation}\label{eq:CauchyEuler}
x_a(t)=\begin{cases}
-(t-1)^2, & \textrm{if }t\le 1,\\
0, & \textrm{if }1\le t\le a,\\
(t-a)^2, & \textrm{if }t\ge a,
\end{cases}
\end{equation}
as $a$ varies in $[1,+\infty[$ (see \textbf{Figure \ref{fig4.1}}). In particular, among the infinite solutions of the MDE, we have $\boldsymbol\mu^1=\{\mu^1_t\}_{t\in[0,T]}$, with $\mu^1_t=\delta_{x_1(t)}$, and $\boldsymbol\mu^2=\{\mu^2_t\}_{t\in[0,T]}$, with $\mu^2_t=\delta_{x_\infty(t)}$.

Computing the LAS scheme for $N=1$ we get:
\begin{enumerate}
\item $\mu_0=\delta_{-1}$, hence $v(-1)=2$ which belongs to the velocity grid;
\item so we get $\mu^{N=1}_{|t=1}=\delta_{(-1+2)}=\delta_1$, hence $v(1)=2$ which belongs to the velocity grid;
\item so we get $\mu^{N=1}_{|t=2}=\delta_{(1+2)}=\delta_3$, hence $v(3)=2\sqrt{3}$ which does not belong to the velocity grid. Since $3<2\sqrt{3}<4$, then the point in the discretized space of velocities for $N=1$ such that $m_{ij}^v(V[\mu])\neq 0$ is $v_j=3$;
\item so we get $\mu^{N=1}_{|t=3}=\delta_{3+3}=\delta_6$, and so on.
\end{enumerate}
For $N=2$ and $N=3$, by performing similar computations we obtain the trajectories as represented in \textbf{Figure \ref{fig4.2}}. We can show that the LAS scheme converges to $\boldsymbol\mu^2$, and thus so does the semidiscrete Lagrangian scheme, up to subsequences.
Moreover we notice that, due to the atomic nature of $V$ over the fibers $T_x\R$, the mean velocity scheme \eqref{MVP} coincides with the semidiscrete Lagrangian one \eqref{DTscheme}. Thus, all the three schemes converge, up to subsequences to the same solution $\boldsymbol\mu^2$. Finally we point out that the semidiscrete Lagrangian scheme corresponds to the Euler method for the underlying ODE. We also notice that in our case, for all $N\in\mathbb N$ the grid intersects the critical point $(t,x)=(1,0)$ where we loose local Lipschtizianity of the vector field. If we perform a perturbation of the grid, shifting it w.r.t. the critical point, then the schemes will converge to $\boldsymbol\mu^1$, up to subsequences. 
\begin{figure}[h!]
\centering
\includegraphics[scale=0.65]{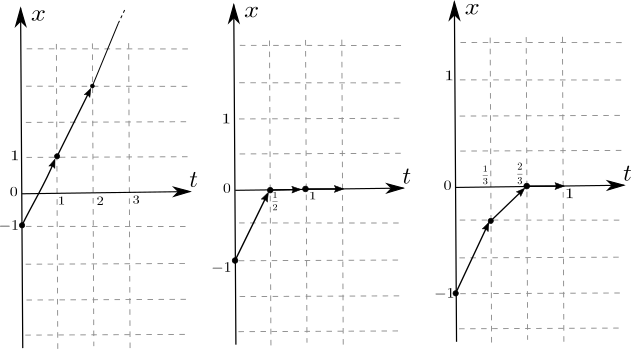}
\centering
\caption{LAS scheme for N= 1, 2, 3. }
\label{fig4.2}
\end{figure}
\end{example}

\medskip

The lack of uniqueness for the notion of weak solution given in Definition \ref{def:sol} and exploited in the examples is not surprising, as already observed in Remark \ref{lack_uni}. 
On the other side, by Proposition \ref{prop:equivMDE-CE}, if the mean velocity field is enough regular, the theory in \cite{AGS} would grant us the uniqueness of a solution of the MDE as push-forward of the initial condition.

\bigskip

\textbf{Acknowledgements}: G.C. thanks SBAI, ``Sapienza'' Universit{\`a} di Roma for its valuable hospitality during the preparation of the paper. G.C. is also indebted with University of Pavia where this research has been partially carried out, in particular G.C. has been supported by Cariplo Foundation and Regione Lombardia via project \emph{Variational Evolution Problems and Optimal Transport}, and by MIUR PRIN 2015 project \emph{Calculus of Variations}, together with FAR funds of the Department of Mathematics of the University of Pavia. G.C. thanks also the support of the INdAM-GNAMPA Project 2019 \emph{Optimal transport for dynamics with interaction (``Trasporto ottimo per dinamiche con interazione'')}. B.P. acknowledges the support of the National Science Foundation under the CPS SynergyGrant No. CNS-1837481.\\
The authors are grateful to the anonymous reviewers for their interesting comments and remarks which pushed to provide greater clarity to the presentation of the work.

%
%

\end{document}